\documentclass[a4paper,11pt,reqno]{amsart}
\usepackage{amsmath,enumerate}
\usepackage{amsthm}

\usepackage{graphicx}
\usepackage{amssymb}

\usepackage{color}

\usepackage{url}

\usepackage[italian,english]{babel}

\usepackage[utf8x]{inputenc}
\usepackage{fancyhdr}

\usepackage{calc}

\usepackage[text={6in,8.6in},centering]{geometry}

\usepackage{srcltx}

\usepackage[T1]{fontenc}

\addto\captionsitalian{}

\newtheorem{thm}{Theorem}
\newtheorem{lem}[thm]{Lemma}

\newtheorem{den}[thm]{Definition}
\newtheorem{oss}[thm]{Remark}
\newtheorem{pro}[thm]{Proposition}

\numberwithin{thm}{section}

\numberwithin{equation}{section}

\newcommand{\R}{\mathbb{R}}
\newcommand{\LL}{{\rm L}}
\newcommand{\HH}{{\rm H}}

\title[The fractional Laplacian in power-weighted $\LL^p$ spaces]{The fractional Laplacian in power-weighted $\LL^p$ spaces: integration-by-parts formulas and self-adjointness}

\author{Matteo Muratori}

\address{Dipartimento di Matematica ``F. Casorati'', Universit\`a degli Studi di Pavia, via A.~Ferrata 5, 27100 Pavia, Italy}

\email{matteo.muratori@unipv.it}

\begin{document}

\begin{abstract}
We consider the fractional Laplacian operator $(-\Delta)^s$  (let $ s \in (0,1) $) on Euclidean space and investigate the validity of the classical integration-by-parts formula that connects the $ \LL^2(\mathbb{R}^d) $ scalar product between a function and its fractional Laplacian to the nonlocal norm of the fractional Sobolev space $ \dot{\HH}^s(\mathbb{R}^d) $. More precisely, we focus on functions belonging to some weighted $ \LL^2 $ space whose fractional Laplacian belongs to another weighted $ \LL^2 $ space: we prove and disprove the validity of the integration-by-parts formula depending on the behaviour of the weight $ \rho(x) $ at infinity. The latter is assumed to be like a power both near the origin and at infinity (the two powers being possibly different). Our results have direct consequences for the self-adjointness of the linear operator formally given by $ \rho^{-1}(-\Delta)^s $. The generality of the techniques developed allows us to deal with weighted $ \LL^p $ spaces as well.
\end{abstract}

\maketitle

\section{Introduction}\label{sec: intro}
Given $ d \in \mathbb{N} $ and any $ s \in (0,1) $, the fractional Laplacian $ (-\Delta)^s $ in $ \mathbb{R}^d $ is a nonlocal operator defined on test functions by
\begin{equation*}\label{eq: frac-test}
(-\Delta)^s(\phi)(x) := C_{d,s} \ p.v. \int_{\mathbb{R}^d} \frac{\phi(x)-\phi(y)}{|x-y|^{d+2s}} \, \mathrm{d}y \quad \forall x \in \mathbb{R}^d \, , \ \, \forall \phi \in \mathcal{D}(\mathbb{R}^d) \, ,
\end{equation*}
where $p.v.$ denotes the \emph{principal value} of the integral about $x$ and $ C_{d,s} $ is a suitable positive constant depending only on $d$ and $s$, such that $ \lim_{s \to 1^-} (-\Delta)^s(\phi) = -\Delta \phi $ (see for instance \cite[Sections 3, 4]{NPV}). An alternative representation of $ (-\Delta)^s $ is the one involving the celebrated extension of Caffarelli and Silvestre \cite{CS}, where the fractional Laplacian of $\phi$ is seen as the trace of the normal derivative of the harmonic extension of $ \phi $ in the upper half-plane (at least for $ s=\tfrac12 $, while for a general $ s \in (0,1) $ one has to introduce a suitable degenerate or singular elliptic operator). Even though it has proved to be a very powerful tool in dealing with issues related to the fractional Laplacian, we shall no further consider the aforementioned extension, since our arguments need not take advantage of it.

\medskip
A Sobolev space naturally associated with the fractional Laplacian is $ \dot{\HH}^s(\mathbb{R}^d) $, namely the closure of $ \mathcal{D}(\mathbb{R}^d) $ endowed with~the norm 
$$  \left\|  \phi  \right\|_{\dot{\HH}^s(\mathbb{R}^d)} := \left\| (-\Delta)^{s/2}(\phi) \right\|_{\LL^2(\mathbb{R}^d)}  \quad \forall \phi \in \mathcal{D}(\mathbb{R}^d) \, . $$  
A well-known result (see \cite[Proposition 3.6]{NPV}) asserts that
\begin{equation}\label{eq: equiv-nonloc}
\left\|  \phi  \right\|_{\dot{\HH}^s(\mathbb{R}^d)}^2 = \frac{C_{d,s}}{2} \int_{\mathbb{R}^d} \int_{\mathbb{R}^d} \frac{\left(\phi(x) - \phi(y)\right)^2}{|x-y|^{d+2s}}  \, \mathrm{d}x \mathrm{d}y \quad \forall \phi \in \mathcal{D}(\mathbb{R}^d) \, , 
\end{equation}
so that we can equivalently define $ \dot{\HH}^s(\mathbb{R}^d) $ by means of the nonlocal (squared) norm appearing in the r.h.s.~of \eqref{eq: equiv-nonloc}. Let us point out that by $ \HH^s(\mathbb{R}^d) $ one usually means the space of functions $ v \in \LL^2(\mathbb{R}^d) $ such that $ \| v \|_{\dot{\HH}^s(\mathbb{R}^d)} < \infty $, which in fact coincides with $ \LL^2(\mathbb{R}^d) \cap \dot{\HH}^s(\mathbb{R}^d) $. However, since below we shall deal with functions belonging to some \emph{weighted} $ \LL^2 $ spaces ($ \LL^p $ in general), throughout the paper we shall never make use of $ \HH^s(\mathbb{R}^d) $.

By means of classical Fourier-transform arguments (we refer again to \cite[Section 3]{NPV}), it is straightforward to show that if $ v \in \LL^2(\mathbb{R}^d) $ and $ (-\Delta)^s(v) \in \LL^2(\mathbb{R}^d) $ (to be understood in the distributional sense), then $ v \in \HH^s(\mathbb{R}^d) $. Moreover, since $ \left\| \cdot  \right\|_{\dot{\HH}^s(\mathbb{R}^d)} $ naturally induces an inner product $ \langle \cdot , \cdot \rangle_{\dot{\HH}^s(\mathbb{R}^d)} $, the following \emph{integration-by-parts} formulas hold:
\begin{equation}\label{eq: parti-intro}
\begin{aligned}
\left\langle v , w \right\rangle_{\dot{\HH}^s(\mathbb{R}^d)} = & \frac{C_{d,s}}{2} \int_{\mathbb{R}^d} \int_{\mathbb{R}^d} \frac{(v(x)-v(y))\,(w(x)-w(y))}{|x-y|^{d+2s}} \, \mathrm{d}x \mathrm{d}y \\
= & \int_{\mathbb{R}^d} (-\Delta)^{s/2}(v)(x) \, (-\Delta)^{s/2}(w)(x) \, \mathrm{d}x  \\
= & \int_{\mathbb{R}^d} v(x) \, (-\Delta)^s(w)(x) \, \mathrm{d}x = \int_{\mathbb{R}^d} (-\Delta)^s(v)(x) \, w(x) \, \mathrm{d}x 
\end{aligned}
\end{equation}
for all $ v,w \in \LL^2(\mathbb{R}^d) $ such that $ (-\Delta)^s(v),(-\Delta)^s(w) \in \LL^2(\mathbb{R}^d) $. We referred to \eqref{eq: parti-intro} as formulas for ``integration by parts'' having in mind the case $s=1$, where the second term is replaced by the $ [\LL^2(\mathbb{R}^d)]^d $ scalar product between gradients. It is worth mentioning that the last line of \eqref{eq: parti-intro} entails the \emph{self-adjointness} of the fractional Laplacian with domain $ \{ v \in \LL^2(\mathbb{R}^d) \! : \, (-\Delta)^s(v) \in \LL^2(\mathbb{R}^d) \} $. We shall resume this point shortly.

\medskip
The main purpose of the paper is to establish (or disprove) the validity of \eqref{eq: parti-intro} in a suitable \emph{weighted} framework. More precisely, let $ \rho(x) $ be a weight (\emph{i.e.}~a nonnegative, measurable function) in $ \mathbb{R}^d $ such that 
\begin{equation}\label{eq: rho-two-powers-intro}
c \, |x|^{-\gamma_0} \le \rho(x) \le C \, |x|^{-\gamma_0} \ \ \textrm{a.e.\ in } B_1 \quad \textrm{and} \quad  c \, |x|^{-\gamma} \le \rho(x) \le C \, |x|^{-\gamma} \ \ \textrm{a.e.\ in } B_1^c
\end{equation}  
for some positive constants $ c<C $ and exponents $ \gamma_0,\gamma \in \R^+ $, where $ B_1 $ denotes the ball of radius one centred at the origin. In other words, we assume that $ \rho(x) $ behaves like a \emph{nonpositive power} both near the origin and at infinity, the two powers being possibly different. We focus on functions $ v,w \in \LL^2(\mathbb{R}^d;\rho(x)\mathrm{d}x) $ such that $ (-\Delta)^s(v),(-\Delta)^s(w) \in \LL^2(\mathbb{R}^d;[\rho(x)]^{-1}\mathrm{d}x) $. Note that, within such class of functions, at least the last line of \eqref{eq: parti-intro} makes sense. In fact we shall prove that \eqref{eq: parti-intro} does hold provided $ \gamma_0 \in [0,d) $ and $ \gamma \in [0,2s] $. The power $ \gamma=2s $ is referred to as \emph{critical} since it is precisely the one corresponding to the \emph{scaling} of the fractional Laplacian, see Section \ref{sect: tools-Lf} below and in particular Lemmas \ref{lem:decay-lap-cutoff} and \ref{lem: stime-termini-cutoff-2}. As we work with weighted Lebesgue spaces, establishing the validity of \eqref{eq: parti-intro} is not a trivial task since we cannot exploit direct Fourier-transform techniques. Indeed our methods of proof will only make use of regularisation-by-mollification and cut-off arguments. In this regard, we devote Section \ref{subsec: density-weight-leb} to prove a result that may also have an independent interest, namely the fact that one can approximate functions in the power-weighted Lebesgue spaces above by means of \emph{standard mollifications} (Theorem \ref{lem: lemma-moll}). This is important to our ends since the mollification operator commutes with translation-invariant operators such as the fractional Laplacian, so that, for instance, a function $ v \in \LL^2(\mathbb{R}^d;\rho(x)\mathrm{d}x) $ with $ (-\Delta)^s(v) \in \LL^2(\mathbb{R}^d;[\rho(x)]^{-1}\mathrm{d}x) $ can be approximated alongside its fractional Laplacian by means of its mollifications (see Proposition \ref{lem: densita-s}). Hence, we start from the validity of \eqref{eq: parti-intro} in $ \mathcal{D}(\mathbb{R}^d) $, mollify $v$ and $w$, cut them off and let the cut-off parameter tend to infinity: in order to make remainder terms vanish, it is essential that $\gamma \le 2s$, \emph{i.e.}~that the power of $ \rho(x) $ at infinity is \emph{subcritical} or at most critical.

Somewhat surprisingly, at least in the case $ d>2s $, we are able to extend the validity of the integration-by-parts formulas to any $ \gamma \in (2s,d] $ as well (that is, to some \emph{supercritical} $\gamma$). However, since cut-off techniques fail, we have to proceed by means of completely different arguments. More precisely, we shall prove that under our assumptions $ v $ and $ w $ coincide with their \emph{Riesz potentials}, namely $ v = \mathsf{I}_{d,s} \ast (-\Delta)^s(v) $ and $ w = \mathsf{I}_{d,s} \ast (-\Delta)^s(w) $, where $ \mathsf{I}_{d,s} $ is the \emph{Riesz kernel} or \emph{Green function} of the fractional Laplacian in $ \mathbb{R}^d $ (see the beginning of Section \ref{sec: supercritical} below and the monograph \cite{Lkof} as a general reference). Loosely speaking, this means that they have much better integrability properties than expected, which is crucial.

Our results can actually be generalised to any $ p \in [2,\infty) $ for $ d \le 2s $ and to any $ p \in [2,2d/(d-2s)] $ for $ d>2s $. Accordingly, the critical power $\gamma=2s$ must be replaced by $ \gamma= d-\tfrac{p}{2}(d-2s) $. The precise statements and introduction of the underlying functional setting are given in Sections \ref{subsec: not}--\ref{subsec: stat}. We preferred to prove the subcritical and supercritical cases separately, since as explained above the techniques are rather different. 

Nevertheless, the case $ p=2 $ is by itself interesting. Indeed, the validity of \eqref{eq: parti-intro} for all $ v,w \in \LL^2(\mathbb{R}^d;\rho(x)\mathrm{d}x) $ such that $ (-\Delta)^s(v),(-\Delta)^s(w) \in \LL^2(\mathbb{R}^d;[\rho(x)]^{-1}\mathrm{d}x) $ is equivalent to the self-adjointness of the linear operator formally given by $ \rho^{-1}(-\Delta)^s $ in $ \LL^2(\mathbb{R}^d;\rho(x)\mathrm{d}x) $. As a consequence, this operator generates a continuous semigroup in $ \LL^2(\mathbb{R}^d;\rho(x)\mathrm{d}x) $, so that the Cauchy problem for the \emph{weighted, fractional heat-type} equation 
\begin{equation}\label{eq: frac-heat}
\begin{cases}
\rho(x) u_t=-(-\Delta)^s(u) & \textrm{in } \mathbb{R}^d \times \mathbb{R}^+ \, , \\
u(x,0)=u_0(x) \in \LL^2(\mathbb{R}^d;\rho(x)\mathrm{d}x) & \textrm{in } \mathbb{R}^d \, ,
\end{cases}
\end{equation}
is well posed. In addition, such semigroup turns out to be Markov and can therefore be extended in a consistent way to a contraction semigroup in $ \LL^p(\mathbb{R}^d;\rho(x)\mathrm{d}x) $ for all $ p \in [1,\infty] $, the latter being analytic for $ p \in (1,\infty) $. The precise statement is provided by Theorem \ref{thm: global-self-adj}.

We finally prove that, still under the assumption $ d>2s $, formulas \eqref{eq: parti-intro} \emph{fail} as soon as $ \gamma>d $ (see Theorem \ref{thm: global}). In particular, we deduce that in this case the operator $ \rho^{-1}(-\Delta)^s $ is \emph{not} self-adjoint in $ \LL^2(\mathbb{R}^d;\rho(x)\mathrm{d}x) $. This is due to the presence of nontrivial constants, since $ \rho \in \LL^1(\mathbb{R}^d) $. Hence, the only set of parameters left undetermined is $ d \le 2s $ and $ \gamma > d-\tfrac{p}{2}(d-2s)$, that is $ d=1 $, $ s \in [1/2,1) $ and $ \gamma > 1+\tfrac{p}{2}(2s-1) $. There our techniques prevent us from establishing whether or not \eqref{eq: parti-intro} holds (see Remark \ref{oss-d-leq-2s}). 

\medskip

The major motivation for investigating the validity of \eqref{eq: parti-intro} in weighted Lebesgue spaces came from \cite{GMP}, a recent paper in which the author and collaborators studied the weighted, \emph{fractional porous medium equation} with initial \emph{measure data}, that is
\begin{equation}\label{eq: frac-pme}
\begin{cases}
\rho(x) u_t=-(-\Delta)^s \! \left( u^m \right) & \textrm{in } \mathbb{R}^d \times \mathbb{R}^+ \, , \\
\rho(x)u(x,0)=\mu & \textrm{in } \mathbb{R}^d \, ,
\end{cases}
\end{equation}
where $ m>1 $ and $ \mu $ is a positive, finite Radon measure on $ \mathbb{R}^d $. In particular, uniqueness is established by suitably adapting a ``duality method'' originally developed in \cite{Pi}, which basically consists in proving that the equation solved by the difference of the Riesz potentials of two possibly different solutions admits zero as its unique solution. In order to apply such method, it is essential to justify rigorously some integration by parts that involve functions belonging to $ \LL^2(\mathbb{R}^d;\rho(x)\mathrm{d}x) $ whose fractional Laplacian belongs to $ \LL^2(\mathbb{R}^d;[\rho(x)]^{-1}\mathrm{d}x) $ (actually in low dimensions one has to cope with analogous issues in weighted $ \LL^p $ spaces, for some $p>2$). Moreover, the well-posedness of \eqref{eq: frac-heat}, which is related to the \emph{dual problem}, is also crucial. The interest in taking measures as initial data comes in turn from the analysis of the asymptotic behaviour of general solutions, see \cite{GMP-asym}. 
 
\medskip 
 
The study of nonlinear diffusion equations involving fractional Laplacians has received an increasing amount of interest recently. In \cite{DQRV1,DQRV2} the authors investigate the fractional porous medium equation (PME from here on) in $ \mathbb{R}^d $ with $ \LL^1(\mathbb{R}^d) $ initial data. A thorough asymptotic analysis is then carried out in \cite{VazBar}. Delicate \emph{a priori} estimates, both from above and below, are the main concern of \cite{BV-frac}. As for the weighted, fractional PME \eqref{eq: frac-pme}, a first well-posedness analysis (for more regular data) is performed by \cite{PT}. Fractional diffusions of porous medium type on \emph{bounded domains} are deeply analysed in \cite{BSV,BV-aprio}. We refer to \cite{VazRec} for an excellent overview of the state of the art in the theory of nonlinear fractional diffusion.
 
Weighted \emph{local} nonlinear diffusions have also been investigated in the last few years. In the series of papers \cite{RV1,RV2,RVK} the authors study the so-called \emph{inhomogeneous} PME in Euclidean space, namely \eqref{eq: frac-pme} with $ s=1 $. They consider regular weights (or \emph{densities}) such that $ \rho(x) \approx |x|^{-\gamma} $ for some $ \gamma>0 $ as $ |x|\to\infty $. It is remarkable that the asymptotics of solutions changes considerably depending on whether $ \gamma $ is lower or greater than the critical value $ \gamma=2 $ (which corresponds to the natural scaling of the Laplacian). The critical case is then addressed in \cite{NR}. A further analysis is carried out by \cite{IS-pme}, where $ \rho(x)=|x|^{-2} $ for all $ x \in \mathbb{R}^d $. For a general theory of weighted PME's we refer the reader \emph{e.g.}~to \cite{GM12,GMPo12}.

\medskip

In the \emph{linear} context, besides the classical reference \cite{Bala}, there are some relatively recent works involving both fractional heat-type equations and weighted, local parabolic equations in $\mathbb{R}^d$. In \cite{BPSV} the authors focus on uniqueness issues for \eqref{eq: frac-heat} with $ \rho \equiv 1 $: they look for a sharp class of positive solutions that can be written as the convolution between their initial datum and the heat kernel. The paper \cite{IS} is the linear counterpart of \cite{IS-pme}. In \cite{EKP} a general weighted, second-order parabolic problem is studied: the density $ \rho $ can depend on time as well, and uniqueness results are discussed as the behaviour of $ \rho(x,t) $ as $ |x|\to\infty $ varies. 

In dimension one, the celebrated paper \cite{Feller1} provides an exhaustive analysis of second-order parabolic equations, possibly degenerate or singular at the boundary, by exploiting the (spectral) theory of one-parameter semigroups due to Hille and Yoshida.
With no claim for completeness, we also quote the more recent works \cite{FGGR,FGGR-2}, where
semigroups generated by linear and quasilinear one-dimensional, weighted, elliptic operators with quite general boundary conditions 
are investigated, and \cite{FMPS-one}, where the authors study a one-dimensional, elliptic operator degenerating of first order at the boundary.
In several space dimensions some degenerate, elliptic operators on domains (with homogeneous Dirichlet boundary conditions) are considered by \cite{FMPP,FMPS}. More precisely, in \cite{FMPP} the authors deal with a weight proportional to the distance to the boundary.
In \cite{FMPS} a thorough analysis of elliptic operators (and of the associated semigroups) whose diffusion coefficients degenerate linearly at the boundary only in tangential directions is performed. In \cite{FGGR-3} similar operators with nonhomogeneous boundary conditions are analysed.

\medskip

The connection between fractional Laplacians and symmetric, $2s$-stable \emph{L\'evy processes} is by now a well-established issue. In this regard, we refer the reader to the nice survey \cite{Val}, where such connection is made apparent by resorting to simple Random walks with jumps. 

The probabilistic interpretation of the fractional Laplacian, and of similar nonlocal diffusion operators, has successfully been exploited to give pointwise bounds on the corresponding heat kernel (\emph{i.e.}~the transition density function of the underlying L\'evy process). 
In \cite{BBCK} the authors study a quite general class of Markov processes of pure jump type, with Dirichlet forms that extend \eqref{eq: equiv-nonloc} and allow for anisotropic processes, obtaining lower and upper bounds for the associated heat kernels. 
By means of real-analytic arguments, two-sided sharp estimates of the heat kernel are given also in \cite{Bog-Grz-2}, where isotropic, unimodal L\'evy processes are considered. In \cite{Chen}, by taking advantage of probabilistic methods, the authors provide two-sided sharp estimates for the heat kernel of the (regional) fractional Laplacian 
in $C^{1,1}$ domains. Less regular domains are dealt with by \cite{Bog-Grz}. 

Let us remark that in the present paper we make no use of stochastic representations for $ \rho^{-1}(-\Delta)^s $. In principle, it is not even clear whether it is possible to associate such an operator with some L\'evy-type process (the role of the weight appears to be non trivial). Nevertheless, the purely analytic problem of providing suitable two-sided bounds for its heat kernel (in terms of $ d $, $s$, $\gamma_0$, $ \gamma $) is also left open.

\subsection{Notations and basic definitions}\label{subsec: not}
For any {nonnegative, nontrivial measurable} function $ \rho $ and measurable set $ \Omega \subset \mathbb{R}^d $, we denote by $ \LL_{\rho}^p(\Omega) $ (let $ p \in [1,\infty) $) the Lebesgue space of all measurable functions $f$ such that
$$  \left\| f \right\|_{\LL_{\rho}^p(\Omega)}^p := \int_{\mathbb{R}^d} \left| f(x) \right|^p \rho(x) \mathrm{d}x < \infty \, . $$
If $ \Omega=\mathbb{R}^d $ we set $ \| f \|_{p,\rho} := \| f \|_{\LL_{\rho}^p(\mathbb{R}^d)} $. Moreover, in the special case of power weights, namely $ \rho(x)=|x|^\lambda $ for some $ \lambda \in \mathbb{R} $, we set $ \LL_{\lambda}^p(\Omega) := \LL_{|x|^{\lambda}}^p(\Omega) $ and $ \| f \|_{p,\lambda} := \| f \|_{p,|x|^\lambda}$. In the non-weighted case $ \lambda=0 $ we use the standard notations $ \LL^p(\Omega) := \LL_0^p(\Omega) $ and $ \| f \|_{p} := \| f \|_{p,0} $.

In the sequel we shall mostly choose $ \Omega = B_r(x_0) $, that is the Euclidean ball of radius $r>0$ centred at $x_0 \in \mathbb{R}^d $, or its complement $ B_r^c(x_0) $. To simplify notation, we adopt the convention $ B_r:=B_r(0) $. 

For weights satisfying appropriate assumptions, we provide a functional space that will be very useful in the sequel.
\begin{den}\label{den: spazio-Xs}
Let $ p \in (1,\infty) $ with $ p^\prime := \tfrac{p}{p-1} $. Suppose that $ \rho $ satisfies \eqref{eq: rho-two-powers-intro} for some $ \gamma_0 \in [0,d) $ and $ \gamma \in [0,d+ps] $. We denote by $X_{p,s,\rho}$ the space of all functions $ v \in \LL^p_{\rho}(\mathbb{R}^d) $ such that $ (-\Delta)^s (v) $ (as a distribution) belongs to $ \LL^{p^\prime}_{\rho^\prime}(\mathbb{R}^d)$, where $ \rho^\prime := \rho^{-(p^\prime-1)} $. 

In the special case $ p=2 $, we set $ X_{s,\rho} := X_{2,s,\rho}$.
\end{den}
According to the above definition, a function $ v \in \LL^p_{\rho}(\mathbb{R}^d)$ belongs to $ X_{p,s,\rho} $ if and only if there exists an element $ \mathcal{V} \in \LL^{p^\prime}_{\rho^\prime}(\mathbb{R}^d) $ such that 
\begin{equation}\label{eq: s-distrib-in-X}
\int_{\mathbb{R}^d} v(x) \, (-\Delta)^s (\phi)(x)  \, \mathrm{d}x = \int_{\mathbb{R}^d} \mathcal{V}(x) \, \phi(x) \, \mathrm{d}x  \ \ \ \forall \phi \in \mathcal{D}(\mathbb{R}^d) \, .
\end{equation} 
We stress that the assumptions $ \gamma_0 \in [0,d) $ and $ \gamma \in [0,d+ps] $ ensure, in particular, that both the left- and the right-hand side of \eqref{eq: s-distrib-in-X} are in fact distributions. 

\subsection{Statement of the main results}\label{subsec: stat}

Our most important contribution to the validity, or otherwise, of the integration-by-parts formulas \eqref{eq: parti-intro} is the following.
\begin{thm} \label{thm: global}
Let either $ d \le 2s $ and $ p \in [2,\infty) $ or $ d>2s $ and $ p \in [ 2, 2d/(d-2s) ] $. Suppose that $ \rho $ satisfies \eqref{eq: rho-two-powers-intro} for some $ \gamma_0 \in [0,d) $ and $ \gamma \in \left[0, d \vee \left( d - \tfrac{p}{2}(d-2s) \right) \right] $. Then formulas \eqref{eq: parti-intro} hold for all $ v,w \in X_{p,s,\rho} $. 

On the other hand, if $ d>2s $ and $ \gamma \in (d,d+ps] $ then formulas \eqref{eq: parti-intro} fail in $ X_{p,s,\rho} $. 
\end{thm}

As mentioned above, Theorem \ref{thm: global} entails some crucial consequences concerning the self-adjointness of the operator $ \rho^{-1}(-\Delta)^s $.

\begin{thm}\label{thm: global-self-adj}
Suppose that $ \rho $ satisfies \eqref{eq: rho-two-powers-intro} for some $ \gamma_0 \in [0,d) $ and $ \gamma \in [0,d \vee 2s] $. Let us define the linear operator $A: D(A) := X_{s,\rho} \subset \LL^2_{\rho}(\mathbb{R}^d) \mapsto \LL^2_{\rho}(\mathbb{R}^d)$ as follows:
\begin{equation*}\label{eq: def-operator-A}
A(f) := \rho^{-1} (-\Delta)^s(f) \quad \forall f \in D(A) \, .
\end{equation*}
Then $A$ is a densely defined, nonnegative self-adjoint operator in $\LL^2_{\rho}(\mathbb{R}^d)$, whose associated quadratic form is
\begin{equation*}
Q(v,v):= \| v \|_{\dot{\HH}^s(\mathbb{R}^d)}^2 = \frac{C_{d,s}}{2} \int_{\mathbb{R}^d} \int_{\mathbb{R}^d} \frac{(v(x)-v(y))^2}{|x-y|^{d+2s}} \, \mathrm{d}x \mathrm{d}y \, ,
\end{equation*}
with domain $ D(Q)=\LL^2_{\rho}(\mathbb{R}^d) \cap \dot{\HH}^s(\mathbb{R}^d)$. Moreover, $Q$ is a Dirichlet form and $ A $ generates a Markov semigroup $S_2(t)$ on $ \LL^2_{\rho}(\mathbb{R}^d)$. In particular, for all $p\in[1,\infty]$ there exists a contraction semigroup $S_p(t)$ on $ \LL^p_{\rho}(\mathbb{R}^d)$, consistent with $S_2(t)$ on $ \LL^2_{\rho}(\mathbb{R}^d)\cap \LL^p_{\rho}(\mathbb{R}^d)$, which is furthermore analytic with a suitable angle $\theta_p>0$ for all $p\in(1,\infty)$.

In the case $d>2s$ and $ \gamma \in (d,d+2s] $, the operator $A$ is no more self-adjoint in $ \LL^2_\rho(\mathbb{R}^d) $.
\end{thm}

\begin{oss}\label{oss-d-leq-2s}\rm
We point out that, upon requiring $ d>2s $, we only exclude the case $ d=1 $ with $ s \in [1/2,1) $. More precisely, in the light of Theorem \ref{thm: global}, in the set of parameters $ d=1 $, $ s \in [1/2,1) $, $ p \in [2,\infty) $ and $ 1+\tfrac{p}{2}(2s-1) < \gamma \le 1+ps $, the validity (or the failure) of \eqref{eq: parti-intro} in $ X_{p,s,\rho} $ is left as an open problem, since for such choices neither cut-off nor potential techniques work (see Sections \ref{sect: SAIP}--\ref{sec: supercritical}).
\end{oss}

\subsection{Organization of the paper}\label{subsec: op}
Section \ref{subsec: density-weight-leb} is entirely devoted to the proof of the fact that mollifications are dense in the weighted $\LL^p$ spaces we consider (Theorem \ref{lem: lemma-moll}). We then briefly show that our assumptions on the weight for such a result to hold are to some extent sharp (Remark \ref{oss: sharp-moll}). In Section \ref{sect: tools-Lf} first we collect some straightforward decay and scaling properties of fractional Laplacians of test functions (Lemmas \ref{lem:decay-lap-1}--\ref{lem: stime-termini-cutoff-2}), then we establish some fundamental intermediate steps, involving the space $ X_{p,s,\rho} $, which are essential to prove the integration-by-parts formulas (Proposition \ref{lem: densita-s} and Lemmas \ref{lem: laplaciano-classico-laplaciano-distrib}, \ref{lem: stime-integrali-cutoff}). The proof of Theorems \ref{thm: global}--\ref{thm: global-self-adj} is split between Sections \ref{sect: SAIP} and \ref{sec: supercritical}. In Section \ref{sect: SAIP}, after providing a continuous-embedding result (Lemma \ref{lem: iniezione-continua}), we give the proof of the validity of the integration-by-parts formulas for subcritical-critical powers. The special case $p=2$ is discussed afterwards. Finally, Section \ref{sec: supercritical} deals with supercritical powers under the assumption $ d>2s $: by means of potential techniques, upon establishing some preliminary results (Lemmas \ref{lem: decay-pot}--\ref{lem: coinc-pot}), we prove and disprove the validity of the integration-by-parts formulas.

\section{Density of mollifications in power-weighted $ \LL^p $ spaces}\label{subsec: density-weight-leb}
In the following we shall assume that $ \rho $ is a weight that behaves like a power, not necessarily negative, both near the origin and at infinity, namely that 
\begin{equation}\label{eq: rho-two-powers}
c \, |x|^{\lambda} \le \rho(x) \le C \, |x|^{\lambda} \ \ \textrm{a.e.\ in } B_1 \quad \textrm{and} \quad  c \, |x|^{\Lambda} \le \rho(x) \le C \, |x|^{\Lambda} \ \ \textrm{a.e.\ in } B_1^c
\end{equation}  
for some positive constants $ c<C $ and exponents $ \lambda,\Lambda \in \R $. 

Our aim in the present section is to show that the standard mollification of a function $ f \in \LL^p_\rho(\mathbb{R}^d) $ converges to $f$ in $ \LL^p_\rho(\mathbb{R}^d) $, under suitable assumptions on $ p $ and $ \lambda $. This result will frequently be exploited through Sections \ref{sect: tools-Lf}--\ref{sec: supercritical}, but we believe it can also have an independent interest.
\begin{thm} \label{lem: lemma-moll}
Let $p \in (1,\infty) $. Suppose that \eqref{eq: rho-two-powers} holds for some $\lambda \in \left(-d, (p-1) d \right)$ and $ \Lambda \in \R $. Let $f \in \LL^p_{\rho}(\mathbb{R}^d)$ and consider the mollification 
\begin{equation}\label{eq: def-mollificazione}
f_\varepsilon(x) := \int_{\mathbb{R}^d} \eta_\varepsilon(x-y) \, f(y) \, \mathrm{d}y \ \ \ \forall x \in \mathbb{R}^d \, ,
\end{equation}
where 
$$ \eta_\varepsilon(x):= {\varepsilon^{-d}} \, \eta\left( \frac{x}{\varepsilon} \right) \quad \forall x \in \mathbb{R}^d \, , \ \, \forall \varepsilon>0 \, ,  $$
and $\eta$ is a nonnegative, regular function supported in $B_1$, such that $\int_{\mathbb{R}^d}\eta(x) \, \mathrm{d}x = 1$. Then, $f_\varepsilon \in C^\infty(\mathbb{R}^d) \cap \LL^p_{\rho}(\mathbb{R}^d)$ and
\begin{equation} \label{eq: approx-L2}
\lim_{\varepsilon \rightarrow 0} \left\| f_\varepsilon - f \right\|_{p,\rho} = 0 \, . 
\end{equation}
\begin{proof}
To simplify readability we shall only deal with the model case $ \rho(x) = |x|^{\lambda} $. Minor modifications are required to deal with a more general weight as in the statement, which will be discussed in the end of the proof. 

In order to give sense to \eqref{eq: def-mollificazione} and to prove that $f_\varepsilon \in C^\infty(\mathbb{R}^d)$, we first need to show that $f \in \LL^1_{\rm loc}(\mathbb{R}^d)$. To this end, by means of H\"{o}lder's inequality, for any $ R>0$ we get:
\begin{equation} \label{eq: stima-L1-loc} 
\int_{B_R} |f(y)| \, \mathrm{d}y \le \left( \int_{B_R} |y|^{-\frac{\lambda}{p-1}} \, \mathrm{d}y \right)^{\frac{p-1}{p}} \left( \int_{B_R} |f(y)|^p \, |y|^{\lambda} \mathrm{d}y \right)^{\frac1p} \le \frac{ \left| \mathbb{S}_{d-1} \right|^{\frac{p-1}{p}} R^{\frac{(p-1)d-\lambda}{p}} }{\left( d-\frac{\lambda}{p-1} \right)^{\frac{p-1}{p}} } \, \| f \|_{p,\lambda} \, ,
\end{equation} 
where $ \mathbb{S}_{d-1} $ is the unitary $(d-1)$-dimensional sphere. Note that the r.h.s.~of \eqref{eq: stima-L1-loc} is finite in view of the assumption $ \lambda < (p-1)d $. 

The validity of \eqref{eq: approx-L2} is actually implied by the validity of the estimate
\begin{equation} \label{eq: bound-L2}
\left\| f_\varepsilon \right\|_{p,\lambda} \le K \left\| f \right\|_{p,\lambda} \quad \forall f \in \LL^p_{\lambda}(\mathbb{R}^d)
\end{equation}
for a suitable positive constant $K$ independent of $\varepsilon$ and $f$. Indeed, once we have established \eqref{eq: bound-L2}, we can proceed in a standard way. First of all, we pick a sequence of functions $\{f_n\} $ that are compactly supported in $ \mathbb{R}^d \setminus \{ 0 \} $ such that
\begin{equation} \label{eq: approx-cont}
\lim_{n \to \infty} \left\| f_n - f \right\|_{p,\lambda} = 0 \, .
\end{equation}
This is always possible: for any given $ n \in \mathbb{N} $ one can consider the truncated functions $ f_n(x):=f(x) \, \chi_{\left\{{1}/{n} \le |x| \le n \right\}} $.
It is plain that each $ {f}_n \in  \LL^p(\mathbb{R}^d) $ is by definition compactly supported in $ \mathbb{R}^d \setminus \{ 0 \} $ and that \eqref{eq: approx-cont} holds. By standard results (see \emph{e.g.}~\cite[Chapters 2, 3]{Ada75} or \cite[Appendix C.4]{Ev98}) the mollification $ (f_n)_\varepsilon $ of $ f_n $ converges to $f_n $ in $ \LL^p(\mathbb{R}^d)$ as $ \varepsilon \to 0 $. Since $(f_n)_\varepsilon$ is eventually supported in $ B_{2n} \cap B_{{1}/{2n} }^c $ and the weight $ |x|^\lambda $ is equivalent to $1$ in such region, we deduce that
\begin{equation} \label{eq: cont-epsilon}
\lim_{\varepsilon \to 0} \left\| {(f_n)}_\varepsilon - f_n \right\|_{p,\lambda} = 0 \, .
\end{equation}
By using the triangular inequality, the linearity of the mollification operator and \eqref{eq: bound-L2}, we get:
\begin{equation} \label{eq: approx-cont-triangolare-bis}
\left\| f_\varepsilon - f \right\|_{p,\lambda} \le (K+1) \left\| f_n - f \right\|_{p,\lambda} + \left\| {(f_n)}_\varepsilon - f_n \right\|_{p,\lambda} .
\end{equation}
Thanks to \eqref{eq: approx-cont}, for any $ \delta >0 $ we can pick $n_\delta$ so large that $\left\| f_{n_\delta} - f \right\|_{p,\lambda} \le \delta$. By letting $\varepsilon \to 0$ in \eqref{eq: approx-cont-triangolare-bis} with $ n = n_\delta $ and using \eqref{eq: cont-epsilon}, we end up with
\begin{equation*} \label{eq: approx-cont-triangolare-ter}
\limsup_{\varepsilon \rightarrow 0} \left\| f_\varepsilon - f \right\|_{p,\lambda} \le (K+1) \, \delta \, ,
\end{equation*}
whence \eqref{eq: approx-L2} follows from the arbitrariness of $ \delta $.

We are therefore left with proving \eqref{eq: bound-L2}. To this aim, let us split $ \|f_\varepsilon\|_{p,\lambda}^p $ in a convenient way:
\begin{equation} \label{eq: conto-norma-L2}
\left\| f_\varepsilon \right\|_{p,\lambda}^p = \int_{B_{2\varepsilon}} \left| f_\varepsilon (x) \right|^p |x|^\lambda \mathrm{d}x + \int_{B_{2\varepsilon}^c} \left| f_\varepsilon (x) \right|^p |x|^\lambda \mathrm{d}x \, . 
\end{equation}
We shall estimate the two integrals above separately. As for the first one, we have (recall that $ \lambda > -d $):
\begin{equation} \label{eq: conto-norma-L2-a}
\int_{B_{2\varepsilon}} \left|f_\varepsilon(x)\right|^p |x|^\lambda \mathrm{d}x \le  \frac{2^{d+\lambda} \, \varepsilon^{d+\lambda} \left| \mathbb{S}_{d-1} \right|}{d+\lambda} \left\| f_\varepsilon \right\|_{\LL^\infty(B_{2\varepsilon})}^p , 
\end{equation}
where, by virtue of \eqref{eq: def-mollificazione} and \eqref{eq: stima-L1-loc} (with $ R=3\varepsilon $),
\begin{equation} \label{eq: conto-norma-L2-b}
\left\| f_\varepsilon \right\|_{\LL^\infty(B_{2\varepsilon})} \le \frac{\| \eta \|_\infty}{\varepsilon^d} \int_{B_{3\varepsilon}} |f(y)| \, \mathrm{d}y \le 
\frac{3^{\frac{(p-1)d-\lambda}{p}} \left| \mathbb{S}_{d-1} \right|^{\frac{p-1}{p}} \| \eta \|_\infty}{\varepsilon^{\frac{d+\lambda}{p}} \left( d-\frac{\lambda}{p-1} \right)^{\frac{p-1}{p}}} \, \| f \|_{p,\lambda} \, .
\end{equation}
From \eqref{eq: conto-norma-L2-a} and \eqref{eq: conto-norma-L2-b} we obtain
\begin{equation} \label{eq: conto-norma-L2-c}
\int_{B_{2\varepsilon}} \left| f_\varepsilon(x) \right|^p |x|^\lambda \mathrm{d}x \le \frac{2^{d+\lambda} \, 3^{(p-1)d-\lambda} \left| \mathbb{S}_{d-1} \right|^p \| \eta \|_\infty^p }{(d+\lambda)\left( d-\frac{\lambda}{p-1} \right)^{p-1} } \left\| f \right\|_{p,\lambda}^p .
\end{equation}
We now turn to the second integral in the r.h.s.\ of \eqref{eq: conto-norma-L2}. We have:
\begin{equation} \label{eq: conto-norma-L2-d}
\begin{aligned}
\int_{B_{2\varepsilon}^c} \left| f_\varepsilon(x) \right|^p |x|^\lambda \mathrm{d}x \le \int_{\mathbb{R}^d} \left| f(y) \right|^p \left( \int_{B_{2\varepsilon}^c} \eta_\varepsilon(x-y) \, |x|^\lambda \mathrm{d}x \right) \mathrm{d}y \, ,
\end{aligned}
\end{equation}
where we exploited H\"{o}lder's inequality, for any fixed $ x \in B_{2\varepsilon}^c $, with respect to the probability measure $ \eta_\varepsilon(x-y) \mathrm{d}y $. Thanks to \eqref{eq: conto-norma-L2-d}, it is enough to show that there exists a positive constant $K^\prime $, independent of $\varepsilon$, such that 
\begin{equation} \label{eq: stima-peso}
\int_{B_{2\varepsilon}^c} \eta_\varepsilon(x-y) \,  |x|^\lambda \mathrm{d}x  \le K^\prime \, |y|^\lambda \ \ \ \forall y \in \mathbb{R}^d \, .
\end{equation}
To this aim, first of all note that
\begin{equation} \label{eq: stima-peso-2}
\int_{B_{2\varepsilon}^c} \eta_\varepsilon(x-y) \,  |x|^\lambda \mathrm{d}x \le \| \eta \|_\infty \, \frac{\int_{B_{2\varepsilon}^c} \chi_{\{ |x-y| \le \varepsilon \}} \, |x|^\lambda \mathrm{d}x}{\varepsilon^d} \, . 
\end{equation} 
It is apparent that for $ |y| < \varepsilon $ the integral in the r.h.s.\ of \eqref{eq: stima-peso-2} is identically zero, while for $|y| > 2 \varepsilon$ we have:
\begin{equation} \label{eq: stima-peso-3}
\begin{aligned}
\frac{\int_{B_{2\varepsilon}^c} \chi_{\{ |x-y| \le \varepsilon \}} \,  |x|^\lambda \mathrm{d}x}{\varepsilon^d} \le \frac{ \int_{B_{\varepsilon}(y)} |x|^\lambda \, \mathrm{d}x}{\varepsilon^d} \le \frac{\left| \mathbb{S}_{d-1} \right|}{2^\lambda \, d} \, \max{\left\{ 3^\lambda , 1 \right\}} \, |y|^\lambda \, .
\end{aligned}
\end{equation}
Hence, it remains to estimate the r.h.s.\ of \eqref{eq: stima-peso-2} as $y$ varies in $ \overline{B}_{2\varepsilon} \setminus B_{\varepsilon} $. We point out that in such region the following inequality holds:
\begin{equation*} \label{eq: stima-peso-4}
\frac{\int_{B^c_{2\varepsilon}} \chi_{\{ |x-y| \le \varepsilon \}} \, |x|^\lambda \mathrm{d}x}{\varepsilon^d} \le \frac{\int_{ B_{|y|+\varepsilon} \setminus B_{2\varepsilon} } |x|^\lambda \, \mathrm{d}x}{\varepsilon^d} = \frac{\left| \mathbb{S}_{d-1} \right|}{d+\lambda} \, \frac{ (|y| +\varepsilon)^{d+\lambda}-2^{d+\lambda} \, \varepsilon^{d+\lambda}}{\varepsilon^d} \, .
\end{equation*}
It is then direct to see that there exists a positive constant $M$, independent of $\varepsilon$, such that
\begin{equation} \label{eq: stima-peso-5}
\frac{ (|y|+\varepsilon)^{d+\lambda}-2^{d+\lambda}\, \varepsilon^{d+\lambda}}{\varepsilon^d} \le M \, |y|^\lambda \ \ \ \forall y \in \overline{B}_{2\varepsilon} \setminus B_\varepsilon \, .
\end{equation}
Hence, by
gathering \eqref{eq: stima-peso-2}--\eqref{eq: stima-peso-5}, we can deduce that \eqref{eq: stima-peso} does hold with
$$ K^\prime = \left| \mathbb{S}_{d-1} \right| \| \eta \|_\infty \, \max{\left\{ \frac{\max{\left\{ 3^\lambda , 1 \right\}}}{2^\lambda \, d} , \frac{M}{d+\lambda} \right\}} \, ,  $$
whence inequality \eqref{eq: bound-L2} follows in view of \eqref{eq: conto-norma-L2-c} and \eqref{eq: conto-norma-L2-d}, which completes the proof.

In order to handle a weight $ \rho $ whose power-type behaviours near the origin and at infinity are different, one can split $f$ in the sum $ f=f_1+f_2 $, with $ f_1:=f \chi_{B_1} $ and $ f_2:=f \chi_{B_1^c} $. By linearity, $ f_\varepsilon = (f_1)_\varepsilon + (f_2)_\varepsilon $; it is therefore enough to show that \eqref{eq: approx-L2} holds for $f_1$ and $f_2$ separately. As concerns $ f_1 $, since the latter and its mollifications are (eventually) supported in $B_{3/2} $, one can modify $ \rho(x) $ so that it behaves like $ |x|^\lambda $ also in $B_{3/2}^c $ and then apply the first part of the proof. Similarly, because $ f_2 $ and its mollifications are eventually supported in $ B_{1/2}^c $, the validity of the analogue of \eqref{eq: bound-L2} (and so of \eqref{eq: approx-L2}) is now implied by the validity of \eqref{eq: stima-peso} in the region $ \{ |y| > 1/2 \} $, which holds for all $ \lambda = \Lambda \in \mathbb{R} $ in view of \eqref{eq: stima-peso-3}.
\end{proof}
\end{thm}

\begin{oss}\label{oss: sharp-moll} \rm
Note that the above assumption $ \lambda \in \left(-d, (p-1) d \right) $ is necessary. Indeed, consider the following function:
$$ g(x) := \frac{\chi_{B_{1/2}}(x)}{|x|^d \log|x|} \quad \forall x \in \mathbb{R}^d \setminus \{ 0 \} \, . $$
It is apparent that $ g \not \in \LL^1_{\rm loc}(\mathbb{R}^d)$ and its mollification $ g_\varepsilon $ is equal to $ -\infty $ in a set of positive measure, for all $ \varepsilon>0 $. However, $ g $ belongs to $ \LL^p_\lambda(\mathbb{R}^d) $ for all $ \lambda \ge (p-1) d $.

As concerns the bound from below over $ \lambda $, let
$$ h(x) := \chi_{B_1} \, |x|^{-\frac{\lambda}{p}} \in \LL^p_\lambda(\mathbb{R}^d)  \, . $$
The mollification $ h_\varepsilon $ is \emph{strictly positive} in a neighbourhood of the origin, for all $ \varepsilon>0 $. In particular if $ \lambda \le -d $ then $ h_\varepsilon \not\in \LL^p_\lambda(\mathbb{R}^d) $, so that \eqref{eq: approx-L2} (with $ f=h $) cannot hold.
\end{oss}

\section{Fractional Laplacians and power-weighted $ \LL^p $ spaces}\label{sect: tools-Lf}
In this section we first discuss some elementary properties concerning the fractional Laplacian, and a similar related operator, applied to standard test functions. We then investigate more in detail the structure of the space $X_{p,s,\rho}$ (recall Definition \ref{den: spazio-Xs}), which represents the precise functional setting where we shall prove or disprove the validity of the integration-by-parts formulas.

\medskip

The proofs of the first two lemmas are omitted, since they follow \emph{e.g.}~by exploiting arguments similar to those used in \cite[Lemma 2.1]{BV-frac}.
\begin{lem} \label{lem:decay-lap-1}
The fractional Laplacian $(-\Delta)^s(\phi)(x) $ of any $ \phi \in \mathcal{D}(\mathbb{R}^d) $ is a regular function that decays (at least) like $|x|^{-d-2s}$ as $|x| \to \infty $.
\end{lem}

\begin{lem} \label{lem:decay-lap-2}
Let $ p \in (1,\infty) $. For any $ \phi \in \mathcal{D}(\mathbb{R}^d) $ the function
\begin{equation*}
l_{p,s}(\phi)(x) := \int_{\mathbb{R}^d}
\frac{\left|\phi(x)-\phi(y)\right|^p}{|x-y|^{d+ps}} \, \mathrm{d}y \ \ \
\forall x \in \mathbb{R}^d
\end{equation*}
is continuous and decays (at least) like $ |x|^{-d-ps} $ as $|x| \to \infty $.

In the special case $ p=2 $ we set $ l_s := l_{2,s} $.
\end{lem}

\begin{lem} \label{lem:decay-lap-cutoff}
For any $R>0$, let $\xi_R$ be the cut-off function
\begin{equation*}
\xi_R(x):=\xi\left(\frac{x}{R}\right)  \ \ \ \forall x \in \mathbb{R}^d \, ,
\end{equation*}
where $ \xi $ is a nonnegative, regular function such that $ \|\xi\|_\infty = 1 $, $ \xi \equiv 1$ in $B_1 $ and $ \xi \equiv 0$ in  $B_2^c $. Then, $ (-\Delta)^s(\xi_R) $ and $ l_{p,s}(\xi_R) $ enjoy the following scaling properties:
\begin{equation}\label{eq: scaling-lap}
(-\Delta)^s(\xi_R)(x)=R^{-2s} \, (-\Delta)^s(\xi)({x}/{R}) \, , \quad l_{p,s}(\xi_R)(x)=R^{-ps} \, l_{p,s}(\xi) ( {x}/{R} ) \quad \forall x \in \mathbb{R}^d \, .
\end{equation}
\begin{proof}
We only show the result for $l_{p,s}$, since the proof for $(-\Delta)^s$ is analogous. Upon letting $ \widetilde{y}=y/R $, one has
\[
l_{p,s}(\xi_R)(x)=  \int_{\mathbb{R}^d} \frac{\left|\xi_R(x)-\xi_R(y)\right|^p}{|x-y|^{d+ps}} \, \mathrm{d}y
= R^{-ps} \int_{\mathbb{R}^d} \frac{| \xi( {x}/{R} )-\xi(\widetilde{y}) |^p}{|x/R-\widetilde{y}|^{d+ps}} \, \mathrm{d}\widetilde{y} = R^{-ps} \, l_{p,s}(\xi)({x}/{R})
\]
for all $ x \in \mathbb{R}^d $.
\end{proof}
\end{lem}

The following lemma displays some consequences of the above properties.
\begin{lem} \label{lem: stime-termini-cutoff-2}
Let $\xi$ and $\xi_R$ be as in Lemma \ref{lem:decay-lap-cutoff}. Let $ q \in [1,\infty) $ and $\gamma \in [0,d+2 q^\prime s)$. Then the norms
\begin{equation}\label{eq: finit-s-lap}
\left\|  |x|^\gamma (-\Delta)^s (\xi) \right\|_{q,-\gamma} \quad \textrm{and} \quad \left\| |x|^\gamma l_s(\xi) \right\|_{q,-\gamma}
\end{equation}
are finite. If in addition $ \gamma \in [0,d+2s] $, then also the norms
\begin{equation}\label{eq: finit-l}
\left\| |x|^\gamma (-\Delta)^s (\xi) \right\|_\infty  \quad \textrm{and}  \quad \left\| |x|^\gamma l_s(\xi) \right\|_\infty
\end{equation}
are finite. Moreover, the following identities hold:
\begin{equation}\label{eq: finit-s-lap-R-q}
\left\| |x|^\gamma (-\Delta)^s \! \left(\xi_R \right) \right\|_{q,-\gamma} = \frac{\left\| |x|^\gamma (-\Delta)^s (\xi) \right\|_{q,-\gamma}}{R^{2s-\gamma - \frac{d-\gamma}{q}}} \, ,
\end{equation}
\begin{equation}\label{eq: finit-s-lap-R}
\left\| |x|^\gamma (-\Delta)^s \! \left(\xi_R \right) \right\|_\infty = \frac{\left\| |x|^\gamma (-\Delta)^s (\xi) \right\|_\infty}{R^{2s-\gamma}} \, ,
\end{equation}
\begin{equation}\label{eq: finit-l-R-q}
\left\| |x|^\gamma l_s \! \left(\xi_R\right) \right\|_{q,-\gamma} = \frac{\left\| |x|^\gamma l_s(\xi) \right\|_{q,-\gamma}}{R^{2s-\gamma - \frac{d-\gamma}{q}}} \, ,
\end{equation}
\begin{equation}\label{eq: finit-l-R}
\left\| |x|^\gamma l_s \! \left(\xi_R\right) \right\|_\infty = \frac{\left\| |x|^\gamma l_s(\xi) \right\|_\infty}{R^{2s-\gamma}} \, .
\end{equation}
\begin{proof}
The finiteness of \eqref{eq: finit-s-lap} and \eqref{eq: finit-l} is ensured by the decay properties of $(-\Delta)^s (\xi)(x)$ and $l_s(\xi)(x)$ recalled by Lemmas \ref{lem:decay-lap-1} and \ref{lem:decay-lap-2}. Identities \eqref{eq: finit-s-lap-R-q}--\eqref{eq: finit-l-R} just follow from \eqref{eq: scaling-lap}.
\end{proof}
\end{lem}

For a function $f$ belonging to $ \LL^1_{\rm loc}(\R^d) \cap \LL^1_{-d-2s}(B_1^c)$, a property that any element of $\LL^p_{\rho}(\mathbb{R}^d)$ (let $ p \in (1,\infty) $) enjoys provided $ \rho $ satisfies \eqref{eq: rho-two-powers} with
\begin{equation*}\label{eq: cond-basic-rho}
\lambda < (p-1)d \quad  \textrm{and}  \quad  \Lambda>-d-2ps \, , 
\end{equation*}
the action
\begin{equation}\label{eq: s-distrib}
\phi \mapsto \int_{\mathbb{R}^d} f(x) \, (-\Delta)^s(\phi)(x) \, \mathrm{d}x  \ \ \ \forall \phi \in \mathcal{D}(\mathbb{R}^d)
\end{equation} 
is indeed an element of $\mathcal{D}^\prime(\mathbb{R}^d)$. This is an immediate consequence of the fact that the notion of convergence of a sequence $\{ \phi_n \} \subset \mathcal{D}(\mathbb{R}^d)$ to $\phi$ in $\mathcal{D}(\mathbb{R}^d)$ implies, in particular, the pointwise convergence of $\{(-\Delta)^s(\phi_n) \}$ to $(-\Delta)^s(\phi)$ and the validity of the bound
\begin{equation*}\label{eq: s-distrib-bound}
\left| (-\Delta)^s(\phi_n)(x) \right| \le K \left( 1+|x| \right)^{-d-2s} \quad \forall x \in \mathbb{R}^d
\end{equation*} 
for a suitable positive constant $K$ independent of $n$ (recall Lemma \ref{lem:decay-lap-1}). Since $f \in \LL^1_{\rm loc}(\R^d) \cap \LL^1_{-d-2s}(B_1^c)$, we can pass to the limit in \eqref{eq: s-distrib} (with $ \phi=\phi_n $) as $ n \to \infty $ by dominated convergence.

In the next lemma we show that, for regular functions having suitable integrability properties at infinity, the distributional fractional Laplacian and the classical one do coincide.
\begin{lem}\label{lem: laplaciano-classico-laplaciano-distrib}
Let $ v \in C^\infty(\mathbb{R}^d) \cap \LL^p_{\Lambda}(B_1^c)$, with $p \in (1,\infty) $ and $\Lambda \ge -d-ps $. Then the classical fractional Laplacian of $v$, defined by
\begin{equation}\label{eq: classical-phi}
(-\Delta)^s(v)(x) := C_{d,s} \  p.v. \int_{\mathbb{R}^d} \frac{v(x)-v(y)}{|x-y|^{d+2s}} \, \mathrm{d}y \quad \forall x \in \mathbb{R}^d \, ,
\end{equation}
is a continuous function which coincides with its distributional fractional Laplacian, in the sense that
\begin{equation}\label{eq: id-classica}
\int_{\mathbb{R}^d} v(x) \, (-\Delta)^s(\phi)(x) \, \mathrm{d}x = \int_{\mathbb{R}^d} (-\Delta)^s (v)(x) \, \phi(x) \, \mathrm{d}x \ \ \ \forall \phi \in \mathcal{D}(\mathbb{R}^d) \, .
\end{equation}
\begin{proof}
To begin with, let us prove that formula \eqref{eq: classical-phi} provides us with a locally bounded function of $x$. To this end, fix $R>0$ and let $x$ vary in $B_R$. It is direct to check that
\begin{equation*}\label{eq: classical-phi-a}
x \mapsto p.v. \int_{B_{2R}} \frac{v(x)-v(y)}{|x-y|^{d+2s}} \, \mathrm{d}y  \ \ \ \forall x \in B_R 
\end{equation*}
is bounded in modulus by a constant (depending on $R$) times $\| \nabla^2 v \|_{\LL^\infty(B_{2R})}$. Moreover, still for $x$ varying in $B_R$, we have:
\begin{equation*}\label{eq: classical-phi-estim-1}
\begin{aligned}
& 2^{-d-2s} \int_{B_{2R}^c} \frac{\left|v(x)-v(y)\right|}{|x-y|^{d+2s}} \, \mathrm{d}y \\
\le & \left\| v \right\|_{ \LL^\infty(B_{R}) } \int_{B_{2R}^c} |y|^{-d-2s} \, \mathrm{d}y  + \| v \|_{\LL^p_\Lambda(B_{2R}^c)} \left( \int_{B_{2R}^c} |y|^{ -\frac{p(d+2s)+\Lambda}{p-1}} \, \mathrm{d}y \right)^{\frac{1}{p^\prime}} ,
\end{aligned}  
\end{equation*}
where $ p^\prime := \tfrac{p}{p-1} $. Note that the r.h.s.~is finite since $v \in C^\infty(\mathbb{R}^d) \cap \LL^p_{\Lambda}(B_1^c) $ with $ \Lambda \ge -d-ps > -d-2ps $. We have therefore proved that $(-\Delta)^s(v)$ is locally bounded. Continuity follows by similar arguments that we omit.

Now we must prove that $(-\Delta)^s(v)$ is in fact the distributional fractional Laplacian of $v$, namely the validity of \eqref{eq: id-classica}. Let us first consider the truncated function $ \xi_R  v \in \mathcal{D}(\mathbb{R}^d) $ (with $\xi_R$ as in Lemma \ref{lem:decay-lap-cutoff}) and observe that, for the latter, the identity
\begin{equation}\label{eq: id-classica-cutoff} 
\int_{\mathbb{R}^d}  \xi_R(x)v(x) \, (-\Delta)^s(\phi)(x) \, \mathrm{d}x  = \int_{\mathbb{R}^d} (-\Delta)^s (\xi_R v)(x) \, \phi(x) \, \mathrm{d}x \ \ \ \forall \phi \in \mathcal{D}(\mathbb{R}^d)
\end{equation}
does hold (recall the related discussion in the Introduction). Using the product formula
\begin{equation}\label{eq: lapll-cut-off-prodotto-phi}
\begin{aligned}
(-\Delta)^s \left( \xi_R v \right) (x) = & \xi_R(x) \, (-\Delta)^s (v)(x) + (-\Delta)^s(\xi_R)(x) \, v(x) \\
& + 2 \, C_{d,s} \int_{\mathbb{R}^d} \frac{\left(v(x) - v(y)  \right) \left( \xi_R(x) - \xi_R(y) \right) }{|x-y|^{d+2s}} \, \mathrm{d}y  
\end{aligned}
\end{equation}
and plugging it in \eqref{eq: id-classica-cutoff}, we get:
\begin{equation}\label{eq: id-classica-cutoff-bis}
\begin{aligned}
& \int_{\mathbb{R}^d} \xi_R(x)v(x) \, (-\Delta)^s(\phi)(x) \, \mathrm{d}x \\
= & \int_{\mathbb{R}^d} \phi(x) \xi_R(x) \, (-\Delta)^s (v) (x)  \, \mathrm{d}x + \int_{\mathbb{R}^d} \phi(x)  v(x) \, (-\Delta)^s(\xi_R)(x) \, \mathrm{d}x \\
& + 2 \, C_{d,s} \int_{\mathbb{R}^d} \int_{\mathbb{R}^d} \phi(x) \, \frac{\left(v(x) - v(y)  \right) \left( \xi_R(x) - \xi_R(y) \right) }{|x-y|^{d+2s}}  \, \mathrm{d}x \mathrm{d}y  \, .
\end{aligned}
\end{equation}
By letting $R \to \infty$ we can pass to the limit safely in the l.h.s.~and in the first term of the r.h.s.\ of \eqref{eq: id-classica-cutoff-bis}, since $v$ is locally regular and integrable at infinity with respect to the weight $ |x|^{-d-2s}$, while $(-\Delta)^s(v)$ is locally bounded as shown above. The second term vanishes: indeed, by taking advantage of \eqref{eq: finit-s-lap-R} (with $\gamma=0 $), we obtain:
\begin{equation*}\label{eq: pass-lim-1}
\int_{\mathbb{R}^d} \left| \phi(x) v(x) \, (-\Delta)^s(\xi_R)(x) \right| \mathrm{d}x \le R^{-2s} \left\| (-\Delta)^s(\xi) \right\|_\infty \int_{\mathbb{R}^d} \left| \phi(x) v(x) \right| \mathrm{d}x \, .
\end{equation*}
In order to handle the last term in the r.h.s.\ of \eqref{eq: id-classica-cutoff-bis}, we have to work a bit more. 
First of all, let us prove that also $l_{p,s}(v)(x)$ is locally bounded (actually continuous). Indeed, for all $R>0$ the function 
\begin{equation}\label{eq: classical-phi-a-p}
x \mapsto \int_{B_{2R}} \frac{\left| v(x)-v(y) \right|^p}{|x-y|^{d+ps}} \, \mathrm{d}y \quad \forall x \in B_R
\end{equation}
is bounded in modulus by a constant (depending on $R$) times $\| \nabla v \|_{\LL^\infty(B_{2R})}^p $. Moreover, 
\begin{equation*}\label{eq: classical-phi-estim-1-p}
2^{-[d-1+p(s+1)]} \int_{B_{2R}^c} \frac{\left|v(x)-v(y)\right|^p}{|x-y|^{d+ps}} \, \mathrm{d}y  \le \left\| v \right\|_{\LL^\infty(B_{R})}^p \int_{B_{2R}^c} |y|^{-d-ps} \, \mathrm{d}y + \| v \|_{\LL^p_{-d-ps}(B_{2R}^c)}^p \, ,
\end{equation*}
and the r.h.s.\ is finite thanks to the assumption $\Lambda \ge -d-ps $. By applying H\"{o}lder's inequality w.r.t.~$\mathrm{d}y $ and using Lemma \ref{lem:decay-lap-cutoff}, it is not difficult to infer the estimate
\begin{equation*}\label{eq: pass-lim-2-p}
\int_{\mathbb{R}^d} \int_{\mathbb{R}^d} \left| \phi(x) \, \frac{\left(v(x) - v(y)  \right) \left( \xi_R(x) - \xi_R(y) \right) }{|x-y|^{d+2s}} \right| \mathrm{d}x \mathrm{d}y 
\le \frac{\| l_{p^\prime,s}(\xi) \|^{\frac{1}{p^\prime}}_{\infty^{\phantom{a}}}}{R^s} \int_{\mathbb{R}^d} | \phi(x) | \left[ l_{p,s}(v)(x) \right]^{\frac{1}{p}} \mathrm{d}x \, .
\end{equation*}
By letting $R \to \infty $ we then deduce that also the last term in the r.h.s.\ of \eqref{eq: id-classica-cutoff-bis} vanishes, so that \eqref{eq: id-classica} is finally proved.
\end{proof}
\end{lem}

Although most of our results will also hold in the case of positive $ \lambda $ and $ \Lambda $, from here on we shall mainly focus on negative powers (\emph{i.e.}~on weights $ \rho $ as in Definition \ref{den: spazio-Xs}), since we aim at considering weights that can be singular near the origin ($ \lambda $ negative), and the matter of the validity of \eqref{eq: parti-intro} becomes less and less trivial the faster $ \rho(x) $ decays as $ |x|\to\infty $ ($ \Lambda $ negative). 

\smallskip

The following proposition shows some useful properties of the space $ X_{p,s,\rho} $.
\begin{pro}\label{lem: densita-s} 
Let $p $, $ p^\prime $, $ \rho $, $ \rho^\prime $ and $ X_{p,s,\rho} $ be as in Definition \ref{den: spazio-Xs}. Then:
\begin{enumerate}[(a)]
\item $\mathcal{D}(\mathbb{R}^d) \subset X_{p,s,\rho}$\,;  \label{X-a}
\item $X_{p,s,\rho}$ endowed with the norm
\begin{equation*}\label{eq: norma-hilb}
\left\| v \right\|_{X_{p,s,\rho}} := \left( \left\| v \right\|_{p,\rho}^2 + \left\| (-\Delta)^s (v) \right\|_{p^\prime,\rho^\prime}^2 \right)^{\frac12} \ \ \ \forall v \in X_{p,s,\rho}
\end{equation*} \label{X-b} 
is a reflexive Banach space (Hilbert if $p=2$)\,;
\item the subspace $C^\infty(\mathbb{R}^d) \cap X_{p,s,\rho}$ is dense in $X_{p,s,\rho}$\,; \label{X-c}
\item the map $ \mathcal{B}: X_{p,s,\rho} \times X_{p,s,\rho} \mapsto \mathbb{R}$, defined by 
\begin{equation*}\label{eq: bilin-Xs}
\mathcal{B}(v,w) := \int_{\mathbb{R}^d} (-\Delta)^s(v)(x) \, w(x) \, \mathrm{d}x  \ \ \ \forall v,w \in X_{p,s,\rho} \, ,
\end{equation*}
is a continuous bilinear form on $X_{p,s,\rho}$\,. \label{X-d}
\end{enumerate} 
\begin{proof}
In order to prove \eqref{X-a} it is enough to check that, for any $\phi \in \mathcal{D}(\mathbb{R}^d)$, we have $\phi \in \LL^p_{\rho}(\mathbb{R}^d)$ and $ (-\Delta)^s(\phi) \in \LL^{p^\prime}_{\rho^\prime}(\mathbb{R}^d)$. This is straightforward: $\rho$ is locally integrable since $ \gamma_0 \in [0,d) $ and $ (-\Delta)^s(\phi)(x)$ is a regular function decaying at least like $|x|^{-d-2s}$ as $|x|\to \infty$ (Lemma \ref{lem:decay-lap-1}), which in particular implies that it belongs to $ \LL^{p^\prime}_{\rho^\prime}(\mathbb{R}^d) $ since $ \gamma \le d+ps $.

As for \eqref{X-b}, let us take a Cauchy sequence $\{ v_n \} \subset X_{p,s,\rho}$. By the definition of $\| \cdot \|_{X_{p,s,\rho}}$ and by the completeness of $ \LL^p_{\rho}(\mathbb{R}^d)$ and $ \LL^{p^\prime}_{\rho^\prime}(\mathbb{R}^d)$, there exist $ v \in \LL^p_{\rho}(\mathbb{R}^d)$ and $\mathcal{V} \in \LL^{p^\prime}_{\rho^\prime}(\mathbb{R}^d)$ such that $ v_n \to v $ in $ \LL^p_\rho(\mathbb{R}^d) $ and $ (-\Delta)^s(v_n) \to \mathcal{V} $ in $ \LL^{p^\prime}_{\rho^\prime}(\mathbb{R}^d)  $ as $ n \to \infty $, respectively. 
Showing the completeness of $X_{p,s,\rho}$ is therefore equivalent to showing that $\mathcal{V}=(-\Delta)^s(v)$, which holds provided we can pass to the limit in the identity
\begin{equation*}\label{eq: s-distrib-in-X-lim}
\int_{\mathbb{R}^d} v_n(x) \, (-\Delta)^s (\phi)(x) \, \mathrm{d}x = \int_{\mathbb{R}^d} (-\Delta)^s (v_n)(x) \, \phi(x)  \, \mathrm{d}x  \ \ \ \forall \phi \in \mathcal{D}(\mathbb{R}^d) \, .
\end{equation*} 
This is indeed the case because $ (-\Delta)^s(\phi) \in \LL^{p^\prime}_{\rho^\prime}(\mathbb{R}^d)$ reads $ \rho^{-1} (-\Delta)^s(\phi) \in \LL^{p^\prime}_{\rho}(\mathbb{R}^d) $, and the same is true for $(-\Delta)^s(v_n)$. The fact that $ X_{s,\rho} $ is Hilbert just follows upon defining the scalar product
\begin{equation*}\label{eq: scalar-hilb}
\left\langle v,w \right\rangle_{X_{s,\rho}} := \int_{\mathbb{R}^d} v(x) \, w(x) \, \rho(x) \mathrm{d}x + \int_{\mathbb{R}^d} (-\Delta)^s(v)(x) \, (-\Delta)^s(w)(x) \, [\rho(x)]^{-1} \mathrm{d}x \ \ \ \forall v,w \in X_{s,\rho} \, .
\end{equation*}
In general $ X_{p,s,\rho} $ is reflexive because so are $ \LL^p_{\rho}(\mathbb{R}^d) $ and $ \LL^{p^\prime}_{\rho^\prime}(\mathbb{R}^d) $.

Now let us deal with \eqref{X-c}. We shall exploit the key result provided by Theorem \ref{lem: lemma-moll}. Thanks to the latter, given any $ v \in X_{p,s,\rho}$ its mollification $v_\varepsilon$ belongs to $C^\infty(\mathbb{R}^d) \cap \LL^p_{\rho}(\mathbb{R}^d)$ and converges to $v$ in $ \LL^p_{\rho}(\mathbb{R}^d)$ as $ \varepsilon \to 0 $. We claim that the fractional Laplacian of $ v_\varepsilon$, which is well defined both in the classical and in the distributional sense in view of Lemma \ref{lem: laplaciano-classico-laplaciano-distrib}, is in fact the mollification of $(-\Delta)^s(v)$, that is
\begin{equation}\label{eq: identita-fondamentale-mollificazione}
(-\Delta)^s (v_\varepsilon) =\left[ (-\Delta)^s (v) \right]_\varepsilon .
\end{equation}
Indeed, for any $ \phi \in \mathcal{D}(\mathbb{R}^d)$ the following identities hold:
\begin{equation*}\label{eq: identita-fondamentale-mollificazione-prova-1}
\begin{aligned}
\int_{\mathbb{R}^d} (-\Delta)^s(\phi)(x) \, v_\varepsilon(x) \, \mathrm{d}x  = & \int_{\mathbb{R}^d}  (-\Delta)^s(\phi)(x) \left( \int_{\mathbb{R}^d} \eta_\varepsilon(x-y) \, v(y) \, \mathrm{d}y  \right) \mathrm{d}x \\
= & \int_{\mathbb{R}^d} \phi(x) \left( \int_{\mathbb{R}^d} \eta_\varepsilon(x-y) 	\, (-\Delta)^s(v)(y) \, \mathrm{d}y \right) \mathrm{d}x \\
= & \int_{\mathbb{R}^d} \phi(x) \left[ (-\Delta)^s (v) \right]_\varepsilon \! (x) \, \mathrm{d}x \, .
\end{aligned}
\end{equation*}
The above exchanges of order of integration are justified by Fubini's Theorem since
$$ 
\int_{\mathbb{R}^d} \int_{\mathbb{R}^d} \left| (-\Delta)^s(\phi)(x) \eta_\varepsilon(x-y) v(y) \right| \mathrm{d}x \mathrm{d}y + \int_{\mathbb{R}^d} \int_{\mathbb{R}^d} \left| \phi(x) \eta_\varepsilon(x-y) (-\Delta)^s(v)(y) \right| \mathrm{d}x \mathrm{d}y < \infty
$$
as a consequence of the fact that $ |v|_\varepsilon , \phi \in \LL^p_{\rho}(\mathbb{R}^d) $ and $ (-\Delta)^s(\phi) , |(-\Delta)^s(v)|_\varepsilon \in \LL^{p^\prime}_{\rho^\prime}(\mathbb{R}^d) $. Having established \eqref{eq: identita-fondamentale-mollificazione} we can use again Theorem \ref{lem: lemma-moll} to deduce that $v_\varepsilon \in C^\infty(\mathbb{R}^d) \cap X_{p,s,\rho}$ and 
$$ 
\lim_{\varepsilon \to 0} \left\| v_\varepsilon - v \right\|_{X_{p,s,\rho}} = \lim_{\varepsilon \to 0} \left( \left\| v_\varepsilon - v \right\|_{p,\rho}^2 + \left\| [(-\Delta)^s(v)]_\varepsilon - (-\Delta)^s(v) \right\|_{p^\prime,\rho^\prime}^2 \right)^{\frac12} = 0 \, ,
$$
which does prove \eqref{X-c}.

The only nontrivial point of \eqref{X-d} is the continuity of $ \mathcal{B} $, which follows as a direct application of H\"{o}lder's inequality w.r.t.~the measure $ \rho(x)\mathrm{d}x $:
\begin{equation*}\label{eq: bilin-Xs-proof}
\left| \mathcal{B}(v,w) \right| \le \left( \int_{\mathbb{R}^d} \left|(-\Delta)^s(v)(x) \right|^{p^\prime} \rho^\prime(x) \mathrm{d}x \right)^{\frac{1}{p^\prime}} \left( \int_{\mathbb{R}^d} \left|w(x)\right|^p \rho(x) \mathrm{d}x \right)^{\frac1p} \le
\left\| v \right\|_{X_{p,s,\rho}} \left\| w \right\|_{X_{p,s,\rho}} .
\end{equation*}
\end{proof}
\end{pro}

We finally establish some crucial integral estimates for functions in $ X_{p,s,\rho} $. 
\begin{lem} \label{lem: stime-integrali-cutoff}
Let either $ p=2 $ and $ \gamma \in [0,d+2s] $ or $ p \in (2,\infty) $ and $ \gamma \in [0,d+ps) $. Suppose that $ \rho $ satisfies \eqref{eq: rho-two-powers-intro} for some $ \gamma_0 \in [0,d) $ and for such a $ \gamma $. Let $ l_s $ be as in Lemma \ref{lem:decay-lap-2} and $ \xi , \xi_R$ be as in Lemma \ref{lem:decay-lap-cutoff}. Let $v_1,v_2 \in C^\infty(\mathbb{R}^d) \cap X_{p,s,\rho} $. Then the integral
\begin{equation}\label{eq: def-IR}
I_R\left( v_{i} \right) := \int_{\mathbb{R}^d} \xi_R^2(x) \left( \int_{\mathbb{R}^d}  \frac{\left( v_{i}(x) - v_{i}(y) \right)^2}{|x-y|^{d+2s}}  \, \mathrm{d}y \right) \mathrm{d}x \ \ \ i=1,2
\end{equation}
is finite for all $R>0$. Moreover, the following estimates hold for all $ R \ge 1 $:
\begin{equation}\label{eq: resto-L2-L2}
\int_{\mathbb{R}^d} \left|  v_1(x) v_2(x) \, \xi_R(x)  (-\Delta)^s (\xi_R)(x) \right| \mathrm{d}x \le \mathcal{T}(R,v_1,v_2) \, ,
\end{equation}
\begin{equation}\label{eq: resto-L2-l}
\int_{\mathbb{R}^d} \left| v_1(y) v_2(y) \right| \left( \int_{\mathbb{R}^d} \frac{\left( \xi_R(x) - \xi_R(y) \right)^2}{|x-y|^{d+2s}} \, \mathrm{d}x \right) \mathrm{d}y \le \mathcal{T}(R,v_1,v_2) \, ,
\end{equation}
\begin{equation}\label{eq: resto-IR-L2}
\int_{\mathbb{R}^d} \int_{\mathbb{R}^d} \left| \xi_R(x) v_1(x) \, \frac{\left(v_2(x) - v_2(y)  \right) \left( \xi_R(x) - \xi_R(y) \right)}{|x-y|^{d+2s}} \right| \mathrm{d}x \mathrm{d}y \le \left[\mathcal{T}(R,v_1,v_1)\right]^{\frac12} \left[ I_R(v_2) \right]^{\frac12} ,
\end{equation}
\begin{equation}\label{eq: resto-IR-L2-bis}
\int_{\mathbb{R}^d} \int_{\mathbb{R}^d} \left| \xi_R(x) v_1(y) \, \frac{\left(v_2(x) - v_2(y)  \right) \left( \xi_R(x) - \xi_R(y) \right)}{|x-y|^{d+2s}} \right| \mathrm{d}x \mathrm{d}y \le \left[\mathcal{T}(R,v_1,v_1)\right]^{\frac12} \left[ I_R(v_2) \right]^{\frac12} ,
\end{equation}
where
\begin{equation}\label{eq: def-code}
\begin{aligned}
& \mathcal{T}(R,v_1,v_2) \\
:= & K \left[ R^{-2s} \, \alpha(R)^{\frac{2\gamma+(p-2)d}{p}} \left\| v_1 \right\|_{p,\rho} \left\| v_2 \right\|_{p,\rho} + R^{-\frac{2d+p(2s-d)-2\gamma}{p}} \left\| v_1 \right\|_{\LL^p_\rho\left( B_{\alpha(R)}^c \right)} \left\| v_2 \right\|_{\LL^p_\rho\left( B_{\alpha(R)}^c \right)} \right] ,
\end{aligned}
\end{equation}
$ \alpha:[1,\infty) \mapsto [1,\infty) $ being any monotone function of $R$ with $ \lim_{R \to \infty} \alpha(R) = \infty $ and $ K$ being a suitable positive constant that depends only on $ d,s,p,\gamma_0,\gamma,c,\xi $.
\begin{proof}
In order to show that $I_R(v_i)$ is finite for all $ R>0 $ note that, since $\xi_R$ is supported in $B_{2R}$, it is enough to prove that the function
\begin{equation*}\label{eq: def-IR-prova-1}
x \mapsto \int_{\mathbb{R}^d}  \frac{\left( v_i(x) - v_i(y) \right)^2}{|x-y|^{d+2s}}  \, \mathrm{d}y
\end{equation*}
stays bounded as $x$ varies in $B_{2R}$. To this end one can proceed exactly as we did in establishing the local boundedness of \eqref{eq: classical-phi-a-p}. We point out that here it is necessary to ask $ \gamma < d+ps $ in the case $ p>2 $.

Now let us deal with estimates \eqref{eq: resto-L2-L2}--\eqref{eq: resto-IR-L2-bis}. Since $\| \xi \|_\infty = 1$, by using \eqref{eq: finit-s-lap-R-q} with $ q=\tfrac{p}{p-2} $, \eqref{eq: finit-s-lap-R} with $ \gamma=0 $, the left-hand inequalities in \eqref{eq: rho-two-powers-intro}, and exploiting a three-point H\"{o}lder's inequality, we find:
\begin{equation}\label{eq: resto-L2-L2-prova-mod-1}
\begin{aligned}
& \int_{\mathbb{R}^d} \left| v_1(x) v_2(x) \, \xi_R(x) (-\Delta)^s (\xi_R)(x) \right| \mathrm{d}x \\
\le & R^{-2s} \left\| v_1 \right\|_{p,\rho} \left\| v_2 \right\|_{p,\rho} \left\| (-\Delta)^s(\xi) \right\|_{\infty} \left\| \rho^{-{2}/{p}} \right\|_{\LL^{\frac{p}{p-2}}\left({B_{\alpha(R)}}\right)} \\
& + R^{-\frac{2d+p(2s-d)-2\gamma}{p}} \, c^{-\frac2p} \left\| v_1 \right\|_{\LL^p_\rho\left( B_{\alpha(R)}^c \right)} \left\| v_2 \right\|_{\LL^p_\rho\left( B_{\alpha(R)}^c \right)} \left\| |x|^\gamma (-\Delta)^s( \xi) \right\|_{\frac{p}{p-2},-\gamma} \, ,
\end{aligned}
\end{equation}
where for $ p=2 $ it is understood that $ \tfrac{p}{p-2}=\infty $. In view of Lemma \ref{lem: stime-termini-cutoff-2} and of the left-hand inequalities in \eqref{eq: rho-two-powers-intro}, it is apparent that \eqref{eq: resto-L2-L2-prova-mod-1} implies \eqref{eq: resto-L2-L2}. Similarly (use \eqref{eq: finit-l-R-q} and \eqref{eq: finit-l-R} instead), we have:
\begin{equation*}\label{resto-L2-l-prova-mod-1}
\begin{aligned}
\int_{\mathbb{R}^d} \left| v_1(y) v_2(y) \right| l_s(\xi_R)(y) \, \mathrm{d}y \le & R^{-2s} \left\| v_1 \right\|_{p,\rho} \left\| v_2 \right\|_{p,\rho} \left\| l_s(\xi) \right\|_{\infty} \left\| \rho^{-{2}/{p}} \right\|_{\LL^{\frac{p}{p-2}}\left({B_{\alpha(R)}}\right)} \\
& + \frac{\left\| v_1 \right\|_{\LL^p_\rho\left( B_{\alpha(R)}^c \right)} \left\| v_2 \right\|_{\LL^p_\rho\left( B_{\alpha(R)}^c \right)} \left\| |x|^\gamma l_s( \xi) \right\|_{\frac{p}{p-2},-\gamma}} {R^{\frac{2d+p(2s-d)-2\gamma}{p}} \, c^{\frac2p} }
\end{aligned}
\end{equation*}
and 
\begin{equation*}\label{eq: resto-IR-L2-1}
\begin{aligned}
& \int_{\mathbb{R}^d} \int_{\mathbb{R}^d} \left| \xi_R(x) v_1(x) \, \frac{\left(v_2(x) - v_2(y)  \right) \left( \xi_R(x) - \xi_R(y) \right)}{|x-y|^{d+2s}} \right| \mathrm{d}x \mathrm{d}y \\
\le &  \left( \frac{\left\| v_1 \right\|_{p,\rho}^2 \left\| l_s(\xi) \right\|_{\infty} \left\| \rho^{-{2}/{p}} \right\|_{\LL^{\frac{p}{p-2}}\left({B_{\alpha(R)}}\right)} }{R^{2s}} + \frac{\left\| v_1 \right\|_{\LL^p_\rho\left( B_{\alpha(R)}^c \right)}^2 \left\| |x|^\gamma l_s(\xi) \right\|_{\frac{p}{p-2},-\gamma}}{R^{\frac{2d+p(2s-d)-2\gamma}{p}} \, c^{\frac2p}} \right)^{\frac12} \left[ I_R(v_2)  \right]^{\frac12} \, ,
\end{aligned}
\end{equation*}
whence \eqref{eq: resto-L2-l} and \eqref{eq: resto-IR-L2}. The proof of \eqref{eq: resto-IR-L2-bis} is completely analogous.
\end{proof}
\end{lem}

\section{Subcritical-critical powers}\label{sect: SAIP}

This section is devoted to the proof of our main results, namely the validity of \eqref{eq: parti-intro} in $ X_{p,s,\rho} $ and a consequent self-adjointness property for the operator $ \rho^{-1} (-\Delta)^s $ in $ \LL^2_\rho(\mathbb{R}^d) $, under the additional assumption that $ \rho $ satisfies \eqref{eq: rho-two-powers-intro} for some $ \gamma $ smaller than or equal to the critical value $ d - \tfrac{p}{2}(d-2s) $. Note that such assumption is restrictive only in the case $d>2s$. To our purposes, we first need a preliminary continuous-embedding result.

\begin{lem} \label{lem: iniezione-continua}
Let either $ d \le 2s $ and $ p \in [2,\infty) $ or $ d>2s $ and $ p \in \left[ 2, {2d}/{(d-2s)} \right] $. Suppose that $ \rho $ satisfies \eqref{eq: rho-two-powers-intro} for some $ \gamma_0 \in [0,d) $ and $ \gamma \in \left[0, d - \tfrac{p}{2}(d-2s) \right] $. Then $X_{p,s,\rho}$ is continuously embedded in $\dot{\HH}^s(\mathbb{R}^d)$ and the following inequality holds:
\begin{equation}\label{eq: iniezione-continua-Hs}
\left\|  v  \right\|_{\dot{\HH}^s(\mathbb{R}^d)}^2 = \frac{C_{d,s}}{2} \int_{\mathbb{R}^d} \int_{\mathbb{R}^d} \frac{(v(x)-v(y))^2}{|x-y|^{d+2s}} \, \mathrm{d}x \mathrm{d}y \le \int_{\mathbb{R}^d} v(x) \, (-\Delta)^s (v)(x) \, \mathrm{d}x \ \ \ \forall v \in X_{p,s,\rho} \, .
\end{equation}
\begin{proof}
We shall first prove \eqref{eq: iniezione-continua-Hs} for the elements of a sequence $ \{ v_n \} \subset C^\infty(\mathbb{R}^d) \cap X_{p,s,\rho} $ converging to $v$ in $ X_{p,s,\rho} $, which exists in view of Proposition \ref{lem: densita-s}--\eqref{X-c}, and then pass to the limit as $ n \to \infty $. To this end, take a family of cut-off functions $\{ \xi_R \}_{R \ge 1} $ as in Lemma \ref{lem:decay-lap-cutoff}. Since $ \xi_R v_n $ belongs to $C^\infty_c(\mathbb{R}^d)$, identity \eqref{eq: parti-intro} with $ v=w=\xi_R v_n $ holds, that is
\begin{equation}\label{eq: lapll-cut-off-prodotto-integrale}
\begin{aligned}
\left\|  \xi_R v_n  \right\|_{\dot{\HH}^s(\mathbb{R}^d)}^2 
 = & \int_{\mathbb{R}^d} \xi_R^2(x) \, v_n(x) \, (-\Delta)^s (v_n)(x) \, \mathrm{d}x + \int_{\mathbb{R}^d} \xi_R(x) \, (-\Delta)^s(\xi_R)(x) \, v_n^2(x) \, \mathrm{d}x \\
& + 2 \, C_{d,s} \int_{\mathbb{R}^d} \int_{\mathbb{R}^d} \xi_R(x) v_n(x) \, \frac{\left(v_n(x) - v_n(y) \right) \left( \xi_R(x) - \xi_R(y) \right)}{|x-y|^{d+2s}} \, \mathrm{d}x \mathrm{d}y
\end{aligned}
\end{equation}
(recall the product formula \eqref{eq: lapll-cut-off-prodotto-phi}). The l.h.s.~of \eqref{eq: lapll-cut-off-prodotto-integrale} reads
\begin{equation}\label{eq: lapll-cut-off-prodotto-Hs}
\begin{aligned}
2 \left\| \xi_R v_n \right\|_{\dot{\HH}^s(\mathbb{R}^d)}^2
= & C_{d,s} \int_{\mathbb{R}^d} \xi_R^2(x) \left( \int_{\mathbb{R}^d}  \frac{\left( v_n(x) - v_n(y) \right)^2}{|x-y|^{d+2s}}  \, \mathrm{d}y \right) \mathrm{d}x \\
 & + C_{d,s} \int_{\mathbb{R}^d} v_n^2(y) \left( \int_{\mathbb{R}^d} \frac{\left( \xi_R(x) - \xi_R(y) \right)^2}{|x-y|^{d+2s}}  \, \mathrm{d}x \right) \mathrm{d}y  \\
 & + 2 \, C_{d,s} \int_{\mathbb{R}^d} \int_{\mathbb{R}^d} \xi_R(x) v_n(y) \, \frac{\left(v_n(x) - v_n(y) \right) \left( \xi_R(x) - \xi_R(y) \right)}{|x-y|^{d+2s}} \, \mathrm{d}x \mathrm{d}y \, .
\end{aligned}
\end{equation}
By exploiting inequality \eqref{eq: resto-IR-L2} from Lemma \ref{lem: stime-integrali-cutoff} with $v_1=v_2=v_n $, and taking advantage of Young's inequality, we estimate the third term in the r.h.s.~of \eqref{eq: lapll-cut-off-prodotto-integrale} as follows:
\begin{equation}\label{eq: stima-terzo-a}
\begin{aligned}
& 2 \, C_{d,s} \int_{\mathbb{R}^d} \int_{\mathbb{R}^d} \left| \xi_R(x) v_n(x) \, \frac{\left(v_n(x) - v_n(y)  \right) \left( \xi_R(x) - \xi_R(y) \right) }{|x-y|^{d+2s}} \right| \mathrm{d}x \mathrm{d}y \\
\le & \delta \, C_{d,s} \, I_R(v_n) + \delta^{-1} \, C_{d,s} \, \mathcal{T}(R,v_n,v_n)
\end{aligned}
\end{equation}
for all $ \delta>0 $, where $ \mathcal{T} $ is defined by \eqref{eq: def-code}. The same can be done for the third term in the r.h.s.\ of \eqref{eq: lapll-cut-off-prodotto-Hs} upon using \eqref{eq: resto-IR-L2-bis}:
\begin{equation}\label{eq: stima-terzo-b}
\begin{aligned}
 & 2 \, C_{d,s} \int_{\mathbb{R}^d} \int_{\mathbb{R}^d} \left| \xi_R(x) v_n(y) \, \frac{\left(v_n(x) - v_n(y)  \right) \left( \xi_R(x) - \xi_R(y) \right) }{|x-y|^{d+2s}} \right| \mathrm{d}x \mathrm{d}y \\
\le & \delta \, C_{d,s} \, I_R(v_n) + \delta^{-1} \, C_{d,s} \, \mathcal{T}(R,v_n,v_n) \, .
\end{aligned}
\end{equation}
Thanks to \eqref{eq: resto-L2-L2} and \eqref{eq: resto-L2-l} with $v_1=v_2=v_n$ we can also estimate the second term in the r.h.s.\ of \eqref{eq: lapll-cut-off-prodotto-integrale} and the second term in the r.h.s.\ of \eqref{eq: lapll-cut-off-prodotto-Hs}:
\begin{equation}\label{eq: stima-secondo-a}
\int_{\mathbb{R}^d} \left| \xi_R(x) \, (-\Delta)^s (\xi_R)(x) \, v_n^2(x) \right| \mathrm{d}x \le \mathcal{T}(R,v_n,v_n) \, ,
\end{equation}
\begin{equation}\label{eq: stima-secondo-b}
C_{d,s} \int_{\mathbb{R}^d} v_n^2(y) \left( \int_{\mathbb{R}^d}  \frac{\left( \xi_R(x) - \xi_R(y) \right)^2}{|x-y|^{d+2s}} \, \mathrm{d}x \right) \mathrm{d}y \le C_{d,s} \, \mathcal{T}(R,v_n,v_n) \, .
\end{equation} 
By combining  \eqref{eq: lapll-cut-off-prodotto-integrale}, \eqref{eq: stima-terzo-a} and \eqref{eq: stima-secondo-a} we thus infer that
\begin{equation}\label{eq: lapll-cut-off-prodotto-integrale-epsilon}
\begin{aligned}
& \left| \int_{\mathbb{R}^d} \xi_R(x) v_n(x) \, (-\Delta)^s ( \xi_R v_n )(x) \, \mathrm{d}x - \int_{\mathbb{R}^d} \xi_R^2(x) \, v_n(x) \, (-\Delta)^s(v_n)(x)  \, \mathrm{d}x \right| \\
\le & \delta \, C_{d,s} \, I_R(v_n) + \left( \delta^{-1}\,C_{d,s}+1 \right) \mathcal{T}(R,v_n,v_n) \, ,
\end{aligned}
\end{equation}
where $ I_R $ is defined by \eqref{eq: def-IR}. Similarly, by gathering \eqref{eq: lapll-cut-off-prodotto-Hs}, \eqref{eq: stima-terzo-b} and \eqref{eq: stima-secondo-b} we get
\begin{equation}\label{eq: lapll-cut-off-prodotto-Hs-epsilon}
\left| 2 \left\| \xi_R v_n \right\|_{\dot{\HH}^s(\mathbb{R}^d)}^2 - C_{d,s} \, I_R(v_n) \right| 
\le \delta \, C_{d,s} \, I_R(v_n) + \left( \delta^{-1} + 1 \right) C_{d,s} \, \mathcal{T}(R,v_n,v_n) \, .
\end{equation}
Hence, \eqref{eq: lapll-cut-off-prodotto-integrale}, \eqref{eq: lapll-cut-off-prodotto-Hs}, \eqref{eq: lapll-cut-off-prodotto-integrale-epsilon} and \eqref{eq: lapll-cut-off-prodotto-Hs-epsilon} yield
\begin{equation*}\label{eq: quasi-immersione-epsilon-R}
\begin{aligned}
C_{d,s} \, I_R(v_n) 
\le & 2 \int_{\mathbb{R}^d} \xi_R^2(x) \, v_n(x) \, (-\Delta)^s(v_n)(x) \, \mathrm{d}x + 3 \, \delta \, C_{d,s} \, I_R(v_n) \\
& + \left( 3 \, \delta^{-1} \, C_{d,s} + C_{d,s} + 2 \right) \mathcal{T}(R,v_n,v_n) \, , 
\end{aligned}
\end{equation*}
that is
\begin{equation}\label{eq: quasi-immersione-epsilon-R-bis}
\frac{1-3\delta}{2} \, C_{d,s} \, I_R(v_n) \le \int_{\mathbb{R}^d} \xi_R^2(x) \, v_n(x) \, (-\Delta)^s(v_n)(x) \, \mathrm{d}x  + \frac{3 \, \delta^{-1} \, C_{d,s} + C_{d,s} + 2}{2} \, \mathcal{T}(R,v_n,v_n) \, .
\end{equation}
It is straightforward to verify that, in view of the hypotheses on $ \gamma $, there holds
\begin{equation*}\label{eq: choice-alpha-vanish}
\lim_{R \to \infty}  \mathcal{T}(R,v_n,v_n) = 0
\end{equation*}
provided $ \alpha(R)=o(R) $ as $ R \to \infty $.
Hence, for any fixed $\delta \in (0,1/3) $, we can pass to the limit in \eqref{eq: quasi-immersione-epsilon-R-bis} as $R \to \infty$ to get, by means \emph{e.g.}~of Fatou's Lemma and dominated convergence, 
\begin{equation}\label{eq: quasi-immersione-epsilon}
\frac{1-3\delta}{2} \, C_{d,s} \int_{\mathbb{R}^d} \int_{\mathbb{R}^d} \frac{(v_n(x)-v_n(y))^2}{|x-y|^{d+2s}}  \, \mathrm{d}x \mathrm{d}y \le \int_{\mathbb{R}^d} v_n(x) \, (-\Delta)^s(v_n)(x) \, \mathrm{d}x \, .
\end{equation}  
Note that the r.h.s.\ of \eqref{eq: quasi-immersione-epsilon} is finite thanks to Proposition \ref{lem: densita-s}--\eqref{X-d}. By letting $\delta \to 0$ we therefore end up with \eqref{eq: iniezione-continua-Hs} for $ v=v_n $; the fact that $v_n$ belongs to $\dot{\HH}^s(\mathbb{R}^d)$ is a consequence of the boundedness of the family $\{ \xi_R v_n \}_{R\ge1}$ in $\dot{\HH}^s(\mathbb{R}^d)$ (we leave it to the reader to check this assertion). A further passage to the limit as $ n \to \infty $ yields the result.
%
\end{proof}
\end{lem}

We are now in position to prove the validity of the integration-by-parts formulas in the case where $ \gamma $ is at most critical.
\begin{proof}[Proof of Theorem \ref{thm: global} (case $  \gamma \le d - \tfrac{p}{2}(d-2s) $)]
We shall proceed along the lines of proof of Lemma \ref{lem: iniezione-continua}, \emph{i.e.}~we start from the validity of the identity
\begin{equation}\label{eq: identita-varepsilon-R}
\left\langle \xi_R v_n , \xi_R w_n \right\rangle_{\dot{\HH}^s(\mathbb{R}^d)} = \int_{\mathbb{R}^d} \xi_R(x) v_n(x) \, (-\Delta)^s (\xi_R w_n)(x) \, \mathrm{d}x \, ,
\end{equation}
where $ \{ v_n \} , \{ w_n \}  \subset C^\infty(\mathbb{R}^d)\cap X_{p,s,\rho} $ are sequences converging to $ v $ and $w$, respectively, in $ X_{p,s,\rho} $. The analogues of \eqref{eq: lapll-cut-off-prodotto-integrale} and \eqref{eq: lapll-cut-off-prodotto-Hs} read
\begin{equation}\label{eq: lapll-cut-off-prodotto-integrale-fg}
\begin{aligned}
& \left\langle \xi_R v_n , \xi_R w_n \right\rangle_{\dot{\HH}^s(\mathbb{R}^d)} \\
= & \int_{\mathbb{R}^d} \xi_R(x) v_n(x) \, (-\Delta)^s ( \xi_R w_n )(x) \, \mathrm{d}x \\
= & \int_{\mathbb{R}^d} \xi_R^2(x) \, v_n(x) \, (-\Delta)^s(w_n)(x) \, \mathrm{d}x + \int_{\mathbb{R}^d} \xi_R(x) \, (-\Delta)^s(\xi_R)(x) \, v_n(x) w_n(x) \, \mathrm{d}x \\
& + 2 \, C_{d,s} \int_{\mathbb{R}^d} \int_{\mathbb{R}^d} \xi_R(x) v_n(x) \, \frac{\left(w_n(x) - w_n(y) \right) \left( \xi_R(x) - \xi_R(y) \right) }{|x-y|^{d+2s}} \, \mathrm{d}x \mathrm{d}y 
\end{aligned}
\end{equation}
and
\begin{equation}\label{eq: lapll-cut-off-prodotto-Hs-fg}
\begin{aligned}
2 \left\langle \xi_R v_n , \xi_R w_n \right\rangle_{\dot{\HH}^s(\mathbb{R}^d)}
= & C_{d,s} \int_{\mathbb{R}^d} \xi_R^2(x) \left( \int_{\mathbb{R}^d}  \frac{\left( v_n(x) - v_n(y) \right)\left(w_n(x) - w_n(y) \right)}{|x-y|^{d+2s}} \, \mathrm{d}y \right) \mathrm{d}x \\
& + C_{d,s} \int_{\mathbb{R}^d} v_n(y) w_n(y) \left( \int_{\mathbb{R}^d}  \frac{\left(\xi_R(x) - \xi_R(y)\right)^2}{|x-y|^{d+2s}}  \, \mathrm{d}x \right) \mathrm{d}y  \\
& + C_{d,s} \int_{\mathbb{R}^d} \int_{\mathbb{R}^d} \xi_R(x) v_n(y) \, \frac{\left(w_n(x) - w_n(y) \right) \left( \xi_R(x) - \xi_R(y) \right) }{|x-y|^{d+2s}}  \, \mathrm{d}x \mathrm{d}y  \\
& + C_{d,s} \int_{\mathbb{R}^d} \int_{\mathbb{R}^d} \xi_R(x) w_n(y) \, \frac{\left(v_n(x) - v_n(y)  \right) \left( \xi_R(x) - \xi_R(y) \right) }{|x-y|^{d+2s}}  \, \mathrm{d}x \mathrm{d}y \, .
\end{aligned}
\end{equation} 
By exploiting Lemma \ref{lem: iniezione-continua} together with the trivial inequality
\begin{equation*}\label{eq: stima-IR-Hs}
C_{d,s} \, I_R(v) \le 2 \left\| v \right\|_{\dot{\HH}^s(\mathbb{R}^d)}^2 \ \ \ \forall v \in X_{p,s,\rho} \, ,
\end{equation*}
we obtain:
\begin{equation}\label{eq: stima-terzo-a-fg}
\begin{aligned}
& 2 \, C_{d,s} \int_{\mathbb{R}^d} \int_{\mathbb{R}^d} \left| \xi_R(x) v_n(x) \, \frac{\left(w_n(x) - w_n(y) \right) \left( \xi_R(x) - \xi_R(y) \right)}{|x-y|^{d+2s}} \right| \mathrm{d}x \mathrm{d}y \\
\le & 2 \sqrt{2} \, C_{d,s}^{\frac12} \left\| w_n \right\|_{\dot{\HH}^s(\mathbb{R}^d)} \left[\mathcal{T}(R,v_n,v_n)\right]^{\frac12} ,
\end{aligned}
\end{equation}
\smallskip
\begin{equation}\label{eq: stima-terzo-b-fg-1}
\begin{aligned}
& C_{d,s} \int_{\mathbb{R}^d} \int_{\mathbb{R}^d} \left| \xi_R(x) v_n(y) \, \frac{\left(w_n(x) - w_n(y) \right) \left( \xi_R(x) - \xi_R(y) \right) }{|x-y|^{d+2s}} \right| \mathrm{d}x \mathrm{d}y \\
\le & \sqrt{2} \, C_{d,s}^{\frac12} \left\| w_n \right\|_{\dot{\HH}^s(\mathbb{R}^d)} \left[\mathcal{T}(R,v_n,v_n)\right]^{\frac12} ,
\end{aligned}
\end{equation}
\smallskip
\begin{equation}\label{eq: stima-terzo-b-fg-2}
\begin{aligned}
& C_{d,s} \int_{\mathbb{R}^d} \int_{\mathbb{R}^d} \left| \xi_R(x) w_n(y) \, \frac{\left(v_n(x) - v_n(y) \right) \left( \xi_R(x) - \xi_R(y) \right) }{|x-y|^{d+2s}} \right| \mathrm{d}x \mathrm{d}y \\
\le & \sqrt{2} \, C_{d,s}^{\frac12} \left\| v_n \right\|_{\dot{\HH}^s(\mathbb{R}^d)} \left[\mathcal{T}(R,w_n,w_n)\right]^{\frac12} ,
\end{aligned}
\end{equation}
\smallskip
\begin{equation}\label{eq: stima-secondo-a-fg}
\int_{\mathbb{R}^d} \left| \xi_R(x) \, (-\Delta)^s (\xi_R)(x) \, v_n(x) w_n(x) \right| \mathrm{d}x \le \mathcal{T}(R,v_n,w_n) \, ,
\end{equation}
\begin{equation}\label{eq: stima-secondo-b-fg}
C_{d,s} \int_{\mathbb{R}^d} \left| v_n(y) w_n(y) \right| \left( \int_{\mathbb{R}^d}  \frac{\left( \xi_R(x) - \xi_R(y) \right)^2}{|x-y|^{d+2s}} \, \mathrm{d}x \right) \mathrm{d}y \le C_{d,s} \, \mathcal{T}(R,v_n,w_n) \, .
\end{equation}
By gathering \eqref{eq: identita-varepsilon-R}--\eqref{eq: lapll-cut-off-prodotto-Hs-fg} and \eqref{eq: stima-terzo-a-fg}--\eqref{eq: stima-secondo-b-fg}, we thus deduce the estimate
\begin{equation}\label{eq: stima-termini-finale-a}
\begin{aligned}
& \left| \frac{C_{d,s}}{2} \int_{\mathbb{R}^d} \xi_R^2(x) \left( \int_{\mathbb{R}^d} \frac{\left( v_n(x) - v_n(y) \right)\left( w_n(x) -w_n(y) \right)}{|x-y|^{d+2s}} \, \mathrm{d}y \right) \mathrm{d}x
 \right. \\
& \left. \ - \int_{\mathbb{R}^d} \xi_R^2(x) \, v_n(x) \, (-\Delta)^s(w_n)(x) \, \mathrm{d}x \right| \\
\le & \sqrt{\frac{C_{d,s}}{2}} \left( 5 \left\| w_n \right\|_{\dot{\HH}^s(\mathbb{R}^d)} \left[\mathcal{T}(R,v_n,v_n)\right]^{\frac12} + \left\| v_n \right\|_{\dot{\HH}^s(\mathbb{R}^d)} \left[\mathcal{T}(R,w_n,w_n)\right]^{\frac12} \right) \\
& + \frac{C_{d,s}+2}{2} \, \mathcal{T}(R,v_n,w_n) \, .
\end{aligned}
\end{equation}
Again, any choice of $ \alpha $ such that $ \alpha(R)=o(R) $ as $ R \to \infty $ implies that the r.h.s.~of \eqref{eq: stima-termini-finale-a} vanishes, so that 
by letting $R \to \infty$ we find the identity
\begin{equation}\label{eq: identita-varepsilon}
\left\langle v_n , w_n \right\rangle_{\dot{\HH}^s(\mathbb{R}^d)} = \int_{\mathbb{R}^d} v_n(x) \, (-\Delta)^s(w_n)(x) \, \mathrm{d}x \, .
\end{equation}
Finally, we let $n \to \infty$ and pass to the limit in \eqref{eq: identita-varepsilon} to get \eqref{eq: parti-intro}: Proposition \ref{lem: densita-s} and Lemma \ref{lem: iniezione-continua} ensure that both the left- and the right-hand side are continuous bilinear forms on $ X_{p,s,\rho} $. 
\end{proof}

By taking advantage of Theorem \ref{thm: global} with $p=2$, we are then able to prove Theorem \ref{thm: global-self-adj} under the additional constraint $ \gamma \le 2s $.
\begin{proof}[Proof of Theorem \ref{thm: global-self-adj} (case $ \gamma \le 2s $)]
It is direct to check that $A$ acts from $D(A)$ to $ \LL^2_{\rho}(\mathbb{R}^d)$, since $ A(f) \in \LL^2_{\rho}(\mathbb{R}^d) $ is equivalent to $ (-\Delta)^s(f) \in \LL^2_{\rho^{-1}}(\mathbb{R}^d)$, which is true by the definition of $ X_{s,\rho}$. The fact that $A$ is densely defined just follows from the inclusion $\mathcal{D}(\mathbb{R}^d) \subset X_{s,\rho} $ (Proposition \ref{lem: densita-s}--\eqref{X-a}). The nonnegativity of $A$ is a trivial consequence of Theorem \ref{thm: global} and the second line of \eqref{eq: parti-intro} with $ v=w=f $.

In order to prove that $ A $ is a symmetric operator, note that the last line of \eqref{eq: parti-intro} can be rewritten as
\begin{equation*}\label{eq: op-symm}
\int_{\mathbb{R}^d} f(x) \, A(g)(x) \, \rho(x) \mathrm{d}x = \int_{\mathbb{R}^d} A(f)(x) \, g(x) \, \rho(x) \mathrm{d}x \quad \forall f,g \in D(A) \, ,
\end{equation*}
which means $D(A) \subset D(A^\ast)$ and $A = A^\ast$ on $D(A)$. Hence, proving that $ A $ is self-adjoint in $ \LL^2_\rho(\mathbb{R}^d) $ amounts to establishing the opposite inclusion $D(A^\ast) \subset D(A)$. To this end we point out that, by the definition of $D(A^\ast)$, one has that $h \in D(A^\ast)$ if and only if $h \in \LL^2_{\rho}(\mathbb{R}^d)$ and there exists a positive constant $M_h$ such that
\begin{equation}\label{eq: op-adj}
\int_{\mathbb{R}^d} h(x) \, A(g)(x) \, \rho(x) \mathrm{d}x \le M_h \left\| g \right\|_{2,\rho} \ \ \ \forall g \in D(A) \, .
\end{equation}
As recalled above, $\mathcal{D}(\mathbb{R}^d) \subset D(A)$: in particular \eqref{eq: op-adj} holds for all $g=\phi \in \mathcal{D}(\mathbb{R}^d)$. Since $\mathcal{D}(\mathbb{R}^d)$ is dense in $ \LL^2_{\rho}(\mathbb{R}^d)$, we can infer the existence of a unique $ \mathcal{E} \in \LL^2_{\rho}(\mathbb{R}^d)$ such that
\begin{equation}\label{eq: op-adj-bis}
\int_{\mathbb{R}^d} h(x) \, A(\phi)(x) \, \rho(x) \mathrm{d}x = \int_{\mathbb{R}^d} h(x) \, (-\Delta)^s(\phi)(x) \, \mathrm{d}x = \int_{\mathbb{R}^d} \mathcal{E}(x) \, \phi(x) \, \rho(x) \mathrm{d}x \ \ \ \forall \phi \in \mathcal{D}(\mathbb{R}^d) \, .
\end{equation}
The last identity in \eqref{eq: op-adj-bis} implies $ \rho \mathcal{E} = (-\Delta)^s(h)$, whence $ h \in X_{s,\rho} = D(A)$.
We have therefore established the inclusion $D(A^\ast) \subset D(A)$, and self-adjointness is proved.

Let us finally deal with the quadratic form $Q$ associated with $A$. Thanks to Theorem \ref{thm: global}, we have that
\begin{equation}\label{eq: quad-form-dom-A}
Q(v,v)= \frac{C_{d,s}}{2} \int_{\mathbb{R}^d} \int_{\mathbb{R}^d} \frac{(v(x)-v(y))^2}{|x-y|^{d+2s}} \, \mathrm{d}x \mathrm{d}y
\end{equation}
for all $ v \in D(A) $. As it is well known (see \emph{e.g.}~\cite[Section 1.2]{Davies}), the domain $D(Q)$ of $ Q $ is precisely the closure of $ D(A) $ endowed with the norm
$$ \left\| v \right\|_{Q} := \sqrt{\left\| v \right\|_{2,\rho^{-1}}^2 + Q(v,v)} = \sqrt{\left\| v \right\|_{2,\rho^{-1}}^2 + \left\| v \right\|_{\dot{\HH}^s(\mathbb{R}^d)}^2} \quad \forall v \in D(A) \, . $$
It is straightforward to check that such closure is nothing but $ \LL^2_{\rho}(\mathbb{R}^d) \cap \dot{\HH}^s(\mathbb{R}^d) $, and that
the quadratic form on $D(Q)=\LL^2_{\rho}(\mathbb{R}^d) \cap \dot{\HH}^s(\mathbb{R}^d)$ is still represented by \eqref{eq: quad-form-dom-A}.

By classical results (see \emph{e.g.}~\cite[Sections 1.3, 1.4]{Davies}), proving that $ A $ generates a Markov semigroup is equivalent to proving that if $ v $ belongs to $D(Q) $ then both $ v \vee 0 $ and $ v \wedge 1 $ belong to $D(Q)$ and satisfy
\begin{equation*}\label{eq: verifica-markov}
Q(v \vee 0,v \vee 0) \le Q(v,v) \ \ \ \textrm{and} \ \ \ Q(v \wedge 1,v \wedge 1) \le Q(v,v) \, ,
\end{equation*}
which is true in view of the characterization of $Q$ given above. The last assertions of the statement then follow from the general theory of symmetric Markov semigroups (see in particular \cite[Theorems 1.4.1, 1.4.2]{Davies}).
\end{proof}

\begin{oss} \label{rem: densita} \rm
\emph{A posteriori}, for $ \rho $ and $ p $ complying with the hypotheses of Theorem \ref{thm: global}, under the additional constraint $ \gamma \le d - \tfrac{p}{2}(d-2s) $, the subspace $ \mathcal{D}(\mathbb{R}^d)$ is dense in $ X_{p,s,\rho} $. Indeed, in view of Proposition \ref{lem: densita-s}--\eqref{X-c}, it is enough to prove that for any $ \varphi \in C^\infty(\mathbb{R}^d) \cap X_{p,s,\rho}$ the cut-off family $ \{ \xi_R \varphi \}_{R \ge 1} $ converges to $ \varphi $ in $X_{p,s,\rho}$ as $R \to \infty$. The only nontrivial point to cope with is the convergence of $\{ (-\Delta)^s(\xi_R \varphi) \} $ to $(-\Delta)^s(\varphi)$ in $\LL^{p^\prime}_{\rho^\prime}(\mathbb{R}^d)$, which can be tackled by arguing as in the proof of Lemma \ref{lem: stime-integrali-cutoff} (we omit the technical details).
\end{oss}

\section{Supercritical powers in the case $ d>2s $}\label{sec: supercritical}

In this section we deal with supercritical values of $ \gamma $, namely $ \gamma > d - \tfrac{p}{2}(d-2s) $, for $ d>2s $ and $ p \in [2,{2d}/{(d-2s)} ] $. \emph{A priori} it is not clear whether or not \eqref{eq: parti-intro} continues to hold. Indeed, from the one hand the space $ X_{p,s,\rho} $ should get larger if one looks at $ \| v \|_{p,\rho} $, on the other hand it should get smaller if one looks at $ \| (-\Delta)^s(v) \|_{p^\prime,\rho^\prime} $. In fact, in agreement with Theorems \ref{thm: global}--\ref{thm: global-self-adj}, we shall see that the actual value that separates the region where \eqref{eq: parti-intro} is valid from the one where it is not is $ \gamma=d $.

\medskip

As mentioned in the Introduction, the assumption $ d>2s $ ensures the existence of the following \emph{Riesz kernel} (or \emph{Green function}) of the fractional Laplacian in $ \mathbb{R}^d $:
\begin{equation*}\label{eq: def-pot}
\mathsf{I}_{d,s}(x):=\frac{\kappa_{d,s}}{|x|^{d-2s}} \quad \forall x \in \mathbb{R}^d \setminus \{ 0 \} \, ,
\end{equation*}
where $ \kappa_{d,s} >0 $ is an explicit positive constant (see \emph{e.g.}~\cite[Chapter I, Section 1]{Lkof}). As it is well known, the Riesz kernel solves
$$  (-\Delta)^s\left( \mathsf{I}_{d,s} \right) = \delta  \quad \textrm{in } \mathbb{R}^d \, , $$
in the sense that
\begin{equation*}\label{eq: def-pot-weak}
\int_{\mathbb{R}^d} \mathsf{I}_{d,s}(x) \, (-\Delta)^s(\phi)(x) \, \mathrm{d}x = \phi(0) \quad \forall \phi \in \mathcal{D}(\mathbb{R}^d) \, .
\end{equation*}
As a consequence, for any $ f $ having good integrability properties, there exists only one solution to the equation $ (-\Delta)^s(F)=f $ in $ \mathbb{R}^d $, which is given by $ F = \mathsf{I}_{d,s} \ast f $. Such a solution is referred to as the \emph{Riesz potential} of $ f $.

\smallskip
By taking advantage of potential techniques, we shall then prove Theorem \ref{thm: global} for supercritical powers. Before, we need some technical lemmas.
%
\begin{lem}\label{lem: decay-pot}
Let $ d>2s $ and $  \phi \in \mathcal{D}(\mathbb{R}^d) $. Then the Riesz potential of $ \phi $, namely $ (\mathsf{I}_{d,s} \ast  \phi)(x) $, is a regular function decaying (at least) like $ |x|^{-d+2s} $ as $ |x| \to \infty $. Moreover, $ |\nabla(\mathsf{I}_{d,s} \ast  \phi)(x)| $ decays (at least) like $ |x|^{-d-1+2s} $ as $ |x| \to \infty $.
\end{lem}
\begin{proof}
These properties are rather standard, hence we omit the proof. As a reference, see \emph{e.g.}~\cite[Lemma 2.2.13]{TDott}.
\end{proof}

\begin{lem}\label{lem: decay-ls-pot}
Let $ d>2s $, $  \phi \in \mathcal{D}(\mathbb{R}^d) $ and $ l_{p,s} $ be defined as in Lemma \ref{lem:decay-lap-2}. Then, under the assumption $ p \in (1 ,{d}/(d-2s) ) $, the function $ l_{p,s}(\mathsf{I}_{d,s} \ast \phi)(x) $ is continuous and decays (at least) like $ |x|^{-p(d-s)} $ as $ |x|\to\infty $.
\end{lem}
\begin{proof}
Again, we shall give a rigorous proof of the decay behaviour only, since continuity easily follows from the properties of $ \mathsf{I}_{d,s} $ and $ \phi $. To this aim, let us set $ \Phi(x):=(\mathsf{I}_{d,s} \ast \phi)(x) $ and observe that
$$ 
\begin{aligned}
l_{p,s}(\Phi)(x) = & \int_{B_{2|x|}^c} \frac{\left| \Phi(x) - \Phi(y) \right|^p}{|x-y|^{d+ps}} \, \mathrm{d}y + \int_{B_{{|x|}/{2}}(x)} \frac{\left| \Phi(x) - \Phi(y) \right|^p}{|x-y|^{d+ps}} \, \mathrm{d}y \\
& + \int_{B_{2|x|} \cap B^c_{{|x|}/{2}}(x) \cap B^c_{{|x|}/{2}} } \frac{\left| \Phi(x) - \Phi(y) \right|^p}{|x-y|^{d+ps}} \, \mathrm{d}y + \int_{B_{{|x|}/{2}}} \frac{\left| \Phi(x) - \Phi(y) \right|^p}{|x-y|^{d+ps}} \, \mathrm{d}y \\
 =: & F_1(x)+F_2(x)+F_3(x)+F_4(x) \, .
\end{aligned}
$$
Since $ |x| \le {|y|}/{2} $ as $ y \in B_{2|x|}^c $, by exploiting the decay behaviour of $ \Phi $ (Lemma \ref{lem: decay-pot}) we easily deduce that $ F_1(x) \le C \, |x|^{-p(d-s)} $ for $ |x| $ large (from here on we denote by $C$ a generic positive constant independent of $ \phi $). As for $ F_2 $, upon noting that $ |\Phi(x)-\Phi(y)|^p \le | \nabla{\Phi}(z) |^p |x-y|^p $ for all $ y \in B_{{|x|}/{2}}(x) $, for some $ z \in B_{{|x|}/{2}}(x) $, and recalling the decay properties of $ |\nabla{\Phi}| $ (Lemma \ref{lem: decay-pot}), we can infer that
$$  F_2(x) \le \max_{z \in B_{{|x|}/{2}}(x)} \left| \nabla{\Phi}(z) \right|^p \int_{B_{{|x|}/{2}}(x)} \frac{1}{|x-y|^{d-p(1-s)}} \, \mathrm{d}y \le C \, |x|^{-p(d-s)} \, . $$
The integral $ F_3 $ can be dealt with as $ F_1 $, with inessential modifications.
Finally, we have:
$$ F_4(x) \le \frac{C}{|x|^{d+ps}} \left( |x|^{d-p(d-2s)} + \int_{B_{{|x|}/{2}}} | \Phi(y)|^p \, \mathrm{d}y \right)  $$
for $ |x| $ large, namely $ F_4(x) \le C \, |x|^{-p(d-s)} $ since $ p < d/(d-2s) $.
\end{proof}

To our ends it is essential to show that, under the running assumptions on $d$, $s$, $ p $ and $\gamma$, all $ v \in X_{p,s,\rho} $ coincide with the Riesz potential of their fractional Laplacian. 
\begin{lem}\label{lem: coinc-pot}
Let $ d>2s $ and $ p \in [ 2, 2d/(d-2s) ] $. Suppose that $ \rho $ satisfies \eqref{eq: rho-two-powers-intro} for some $ \gamma_0 \in [0,d) $ and $ \gamma \in \left( d - \tfrac{p}{2}(d-2s) , d \right] $. Then $ v = \mathsf{I}_{d,s} \ast (-\Delta)^s(v) $ for all $ v \in X_{p,s,\rho} $. 
\end{lem}
\begin{proof}
Given $ v \in X_{p,s,\rho} $ and any test function $ \phi \in \mathcal{D}(\mathbb{R}^d) $, let $ \Phi=\mathsf{I}_{d,s} \ast \phi $. For all $ R \ge 1 $, it is plain that $ \xi_R \Phi $ is also a test function. Hence, by the definition of $ (-\Delta)^s(v) $, in view of the product formula \eqref{eq: lapll-cut-off-prodotto-phi} (with $ v=\Phi $ there) and the fact that $ (-\Delta)^s(\Phi) = \phi $, there holds
\begin{equation}\label{eq: lapll-cut-off-prodotto-phi-bis}
\begin{aligned}
\int_{\mathbb{R}^d} (-\Delta)^s(v)(x) \, \xi_R(x) \, \Phi(x) \, \mathrm{d}x = & \int_{\mathbb{R}^d} v(x) \xi_R(x) \, \phi(x) \, \mathrm{d}x + \int_{\mathbb{R}^d} v(x) \, (-\Delta)^s(\xi_R)(x) \, \Phi(x) \, \mathrm{d}x \\
& + 2 \, C_{d,s} \! \int_{\mathbb{R}^d} v(x) \underbrace{\int_{\mathbb{R}^d} \frac{ \left( \xi_R(x) - \xi_R(y) \right) \left( \Phi (x) - \Phi (y) \right) }{|x-y|^{d+2s}} \, \mathrm{d}y}_{\mathsf{L}_R(x)} \mathrm{d}x \, .
\end{aligned}
\end{equation}
It is apparent that the first term converges to $ \int_{\mathbb{R}^d} v(x) \, \phi(x) \, \mathrm{d}x $ as $ R \to \infty $. As for the second term, upon taking advantage of Lemma \ref{lem:decay-lap-cutoff} and proceeding by means of arguments similar to those exploited in the proof of Lemma \ref{lem: stime-integrali-cutoff} (adopting the same notations), we obtain:
\begin{equation}\label{eq: def-weak-trunc-simpler-critical-modified}
\begin{aligned}
& \left| \int_{\mathbb{R}^d} v(x) \, (-\Delta)^s(\xi_R)(x) \, \Phi(x) \, \mathrm{d}x \right|  \\
\le & \frac{ \| (-\Delta)^s(\xi) \|_\infty \, \| v \|_{p,\rho}}{R^{2s}} \underbrace{\left( \int_{B_{\alpha(R)}} \left| \Phi(x) \right|^{\frac{p}{p-1}} \left[\rho(x)\right]^{-\frac{1}{p-1}} \, \mathrm{d}x \right)^{\frac{p-1}{p}}}_{\sigma(R)}  \\
& +  \frac{Q \left\| v \right\|_{\LL^p_\rho\left(B^c_{\alpha(R)}\right)}}{R^{\frac{d-\gamma}{p}} \, c^{\frac1p}} \underbrace{ \left( \int_{B_{{\alpha(R)}/{R}}^c} \left| (-\Delta)^s(\xi)(x) \right|^{\frac{p}{p-1}} \, |x|^{\frac{\gamma-p(d-2s)}{p-1}} \, \mathrm{d}x  \right)^{\frac{p-1}{p}} }_{\varsigma(R)} \, ,
\end{aligned}
\end{equation} 
where $ Q $ is any positive constant such that $ |\Phi(x)| \le Q \, |x|^{-d+2s} $ for all $ x \in B_1^c $. Let $ \alpha(R) = o(R) $ as $ R \to \infty $. Thanks to Lemma \ref{lem: decay-pot} and the assumptions on $ \rho $, it is not difficult to infer that $ \sigma(R) $ behaves like $ \alpha(R)^{(2ps+\gamma-d)/p} $, like $ \log[\alpha(R)]^{(p-1)/p} $ or tends to a constant as $ R \to \infty $ depending on whether $ \gamma>d-2ps $, $ \gamma=d-2ps $ or $ \gamma<d-2ps $, respectively. Similarly, in view of Lemma \ref{lem:decay-lap-1}, it turns out that $ \varsigma(R) $ tends to a constant, behaves like $ \log[R/\alpha(R)]^{(p-1)/p} $ or like $ [R/\alpha(R)]^{(d-2ps-\gamma)/p} $ as $ R \to \infty $ depending on whether $ \gamma>d-2ps $, $ \gamma=d-2ps $ or $ \gamma<d-2ps $, respectively. In any case, these properties ensure that the r.h.s.~of \eqref{eq: def-weak-trunc-simpler-critical-modified} vanishes as $ R \to \infty $ for all $ \gamma \in \left(d - \frac{p}{2}(d-2s),d \right] $. 

As for the third term in \eqref{eq: lapll-cut-off-prodotto-phi-bis}, we get
\begin{equation}\label{eq: def-weak-trunc-4-simply}
\begin{aligned}
 \left| \int_{\mathbb{R}^d} v(x) \, \mathsf{L}_R(x) \, \mathrm{d}x \right| 
\le & \frac{\left\| l_{q,s}(\xi) \right\|_{\infty^{\phantom{a}}}^{\frac1q}}{R^{s} } \left\| v \right\|_{p,\rho} \underbrace{\left( \int_{B_{\alpha(R)}} \left[ l_{q^\prime,s}(\Phi)(x) \right]^{\frac{p^\prime}{q^\prime}}  \left[\rho(x)\right]^{-\frac{1}{p-1}} \mathrm{d}x \right)^{\frac{p-1}{p}}}_{\sigma_0(R)} \\
& + \frac{Q_0^{\frac{q-1}{q}} \left\| v \right\|_{\LL^p_\rho\left( B^c_{\alpha(R)} \right)}}{R^{\frac{d-\gamma}{p}} \, c^{\frac1p}} \underbrace{\left( \int_{B^c_{{\alpha(R)}/{R}}} \left[ l_{q,s}( \xi )(x) \right]^{\frac{p}{q(p-1)}} |x|^{\frac{\gamma-p(d-s)}{p-1}} \, \mathrm{d}x \right)^{\frac{p-1}{p}}}_{\varsigma_0(R)}
\end{aligned}
\end{equation}
provided $ q>\tfrac{d}{2s} $, upon recalling the definition of $ l_{q,s} $ given in Lemma \ref{lem:decay-lap-2}, using Lemma \ref{lem:decay-lap-cutoff} and arguing similarly to the proof of Lemma \ref{lem: stime-integrali-cutoff}. Here $ Q_0 $ is any positive constant such that $ |l_{q^\prime,s}( \Phi )(x) | \le Q_0 \, |x|^{-q^\prime(d-s)} $ for all $ x \in B_1^c $. Let $ \alpha(R) = o(R) $ as $ R \to \infty $. Thanks to Lemma \ref{lem: decay-ls-pot} ($ q>\tfrac{d}{2s} $ is equivalent to $ q^\prime < d/(d-2s) $) and the assumptions on $ \rho $, one can deduce that $ \sigma_0(R) $ behaves like $ \alpha(R)^{(ps+\gamma-d)/p} $, like $ \log[\alpha(R)]^{(p-1)/p} $ or tends to a constant as $ R\to\infty $ depending on whether $ \gamma > d-ps $, $ \gamma=d-ps $ or $ \gamma < d-ps $, respectively. Similarly, by virtue of Lemma \ref{lem:decay-lap-2}, one infers that $ \varsigma_0(R) $ tends to a constant, behaves like $ \log[R/\alpha(R)]^{(p-1)/p} $ or like $ [R/\alpha(R)]^{(d-ps-\gamma)/p} $ as $ R \to \infty $ depending on whether $ \gamma>d-ps $, $ \gamma=d-ps $ or $ \gamma<d-ps $, respectively. As a consequence, also the r.h.s.~of \eqref{eq: def-weak-trunc-4-simply} vanishes as $ R \to \infty $ for all $ \gamma \in \left(d - \tfrac{p}{2}(d-2s),d\right] $. 

Since $ \rho $ satisfies \eqref{eq: rho-two-powers-intro} for some $ \gamma > d - \tfrac{p}{2}(d-2s) $, from Lemma \ref{lem: decay-pot} it is straightforward to check that $ \Phi $ belongs to $\LL^p_{\rho}(\mathbb{R}^d) $, so that $ (-\Delta)^s (v) \, \Phi \in \LL^1(\mathbb{R}^d) $ and the l.h.s.~of \eqref{eq: lapll-cut-off-prodotto-phi-bis} converges to $ \int_{\mathbb{R}^d} (-\Delta)^s(v)(x) \, \Phi(x) \, \mathrm{d}x $ as $ R \to \infty $. Hence, by letting $ R \to \infty $ in \eqref{eq: lapll-cut-off-prodotto-phi-bis} and using Fubini's Theorem, the assertion follows given the arbitrariness of $ \phi $.
\end{proof}

We are now in position to prove Theorem \ref{thm: global} in the case $ \gamma $ is supercritical. 
\begin{proof}[Proof of Theorem \ref{thm: global} (case $ d>2s $ and $ \gamma >  d-\tfrac{p}{2}(d-2s) $)]
Thanks to Lemma \ref{lem: coinc-pot}, for $ \gamma  $ supercritical and smaller than or equal to $d$ we only need to prove the identity
\begin{equation}\label{eq: def-weak-dom-1}
\int_{\mathbb{R}^d} \left| (-\Delta)^{s/2}\left( \mathsf{I}_{d,s} \ast f \right)\!(x) \right|^2 \mathrm{d}x = \int_{\mathbb{R}^d} \left(\mathsf{I}_{d,s} \ast f \right)\!(x) \, f(x) \, \mathrm{d}x  \quad \forall f \in \LL^{p^\prime}_{\rho^{\prime}}(\mathbb{R}^d) \! : \ \mathsf{I}_{d,s} \ast f \in \LL^{p}_{\rho}(\mathbb{R}^d) \, .
\end{equation} 
To this end, let us consider the Green function of $ (-\Delta)^s $ on $B_R$, namely the unique positive solution $ G_{y,R} $ to 
$$  (-\Delta)_{R}^s \left( G_{y,R} \right) = \delta_y \quad \textrm{in } B_R \, , $$
where by $ (-\Delta)_{R}^s $ we mean the \emph{spectral} fractional Laplacian in $B_R$ and by $ \delta_y $ we denote the Dirac delta centred at $ y \in B_R $. Given $ h \in \LL^2(B_R) $, the unique solution $ u \in \HH^s_0(B_R) $ to 
$$
\begin{cases}
(-\Delta)_{R}^s(u)=h & \textrm{in } {B_R} \, ,   \\
u=0 & \textrm{on } {\partial B_R} \, ,
\end{cases} 
$$
is provided by $ u_R(x)=\int_{B_R} G_{y,R}(x) \, h(y) \, \mathrm{d}y $, in the sense that 
\begin{equation}\label{eq: def-weak-dom-2}
\int_{B_R} (-\Delta)_R^{s/2}(u_R)(x) \, (-\Delta)_R^{s/2}(\psi)(x) \, \mathrm{d}x  =  \int_{B_R} h(x) \, \psi(x) \, \mathrm{d}x \quad \forall \psi \in \HH^s_0(B_R) \, .
\end{equation}
By setting $ h= f_\varepsilon  $, with $ f_\varepsilon $ as in \eqref{eq: def-mollificazione}, and plugging $ \psi=u_R $ in \eqref{eq: def-weak-dom-2}, we find the identity
\begin{equation}\label{eq: def-weak-dom-3}
\int_{B_R} \left| (-\Delta)_R^{s/2}(u_R)(x) \right|^2 \mathrm{d}x  =  \int_{B_R} f_\varepsilon(x) \, u_R(x) \, \mathrm{d}x  \, .
\end{equation}
As it is well known (see \emph{e.g.} \cite[Appendix 11]{BV-aprio} and references therein),
$$ 
\lim_{R \to \infty} G_{y,R}(x) = \mathsf{I}_{d,s}(x-y) \quad \textrm{and} \quad G_{y,R}(x) \le \mathsf{I}_{d,s}(x-y) \quad \forall x \neq y \, , \ \, \forall R>0 \, , 
$$
so that by dominated convergence we can pass to the limit in \eqref{eq: def-weak-dom-3} as $ R \to \infty $ to get 
\begin{equation}\label{eq: def-weak-dom-4}
\int_{\mathbb{R}^d} \left| (-\Delta)^{s/2}( \mathsf{I}_{d,s} \ast f_\varepsilon )(x) \right|^2 \mathrm{d}x  \le  \int_{\mathbb{R}^d} f_\varepsilon(x) \left(\mathsf{I}_{d,s} \ast f_\varepsilon \right)\!(x) \, \mathrm{d}x 
\end{equation}
along with the weak convergence of $ (-\Delta)_R^{s/2}(u_R) $ to $(-\Delta)^{s/2}( \mathsf{I}_{d,s} \ast f_\varepsilon )$ in $ \LL^2(\mathbb{R}^d) $.
Actually, in order to make sure that the r.h.s.~of \eqref{eq: def-weak-dom-4} is finite,
we first have to show that $ \mathsf{I}_{d,s} \ast f_\varepsilon \in \LL^p_\rho(\mathbb{R}^d) $ \emph{as a consequence} of the fact that $ f \in \LL^{p^\prime}_{\rho^\prime}(\mathbb{R}^d) $ (recall that also $ f_\varepsilon \in \LL^{p^\prime}_{\rho^\prime}(\mathbb{R}^d) $ by Theorem \ref{lem: lemma-moll}). To this end, first of all note that $ \mathsf{I}_{d,s} \ast f_\varepsilon \in \LL^\infty_{\rm loc}(\mathbb{R}^d) $. Indeed, for any $ \delta \ge \tfrac{1}{2} $ and all $ x \in B_\delta $, we have:
\begin{equation*}\label{eq: def-weak-dom-cc-intro} 
\begin{aligned}
\left| (\mathsf{I}_{d,s} \ast f_\varepsilon)(x) \right| 
\le \underbrace{\left(\mathsf{I}_{d,s} \ast \left| f_\varepsilon \chi_{B_{2\delta}} \right| \right)\!(x)}_{F(x)} + \, 2^{d-2s} \, \kappa_{d,s} \, C^{\frac{1}{p}} \, \| f_\varepsilon \|_{p^\prime,\rho^\prime} \left( \int_{B_{2\delta}^c} \frac{1}{|y|^{(d-2s)p+\gamma}} \, \mathrm{d}y \right)^{\frac{1}{p}} \, .
\end{aligned}
\end{equation*}
Since $ f_\varepsilon \chi_{B_{2\delta}} \in \LL^1(\mathbb{R}^d) \cap \LL^\infty(\mathbb{R}^d) $, we deduce that $ F \in \LL^\infty(\mathbb{R}^d) $; moreover $ (d-2s)p + \gamma > d $ thanks to the hypotheses on $ \gamma $. As for \emph{global} integrability properties of $ \mathsf{I}_{d,s} \ast f_\varepsilon$, it is readily seen that from $ f_\varepsilon \in \LL^\infty_{\rm loc}(\mathbb{R}^d) \cap \LL^{p^\prime}_{(p^\prime-1)\gamma}(B_1^c)  $ there follows $ f_\varepsilon \in \LL^q(\mathbb{R}^d) $ for all $q$ such that  
$p^\prime d/[d+(p^\prime-1)\gamma] < q \le p^\prime$.
In particular, thanks to \cite[Theorem 1 at p. 119]{Stein}, we infer that $ \mathsf{I}_{d,s} \ast f_\varepsilon \in \LL^r (\mathbb{R}^d)$ for all $ r $ such that
\begin{equation}\label{eq: def-weak-dom-cc-2}
\frac{p^\prime d}{d+(p^\prime-1)\gamma-2 p^\prime s} < r \le \frac{p^\prime d}{d-2p^\prime s} \, .
\end{equation}
Let us point out that in the case where $ d=2p^\prime s $ the upper bound on $ r $ in \eqref{eq: def-weak-dom-cc-2} reads $ r < \infty $, while in the case where $ d < 2p^\prime s $ it reads $ r \le \infty $. Moreover, a routine H\"{o}lder's interpolation shows that $ \mathsf{I}_{d,s} \ast f_\varepsilon \in \LL^p_{-\gamma}(B_1^c) $ provided $ \mathsf{I}_{d,s} \ast f_\varepsilon \in \LL^r(B_1^c) $ for some $ p \le r < {pd}/(d-\gamma) $. Because $ p \in [2,2d/(d-2s)] $ and $ \gamma > d-\tfrac{p}{2}(d-2s) $, it is immediate to check that the latter condition and \eqref{eq: def-weak-dom-cc-2} meet for some $r$. 

Let us denote by $ \HH^s_c(\mathbb{R}^d) $ the space of \emph{compactly supported} functions belonging to $ \dot{\HH}^s(\mathbb{R}^d) $. Given any $ \psi \in \HH^s_c(\mathbb{R}^d) $, it is direct to see that the energy $ \int_{B_R} \big|(-\Delta)_R^{s/2}(\psi)(x)\big|^2 \, \mathrm{d}x $ is eventually nonincreasing w.r.t.~$R$ and converges to $ \int_{\mathbb{R}^d} \big|(-\Delta)^{s/2}(\psi)(x)\big|^2 \, \mathrm{d}x  $ as $ R \to \infty $. In particular, we have that $ (-\Delta)_R^{s/2}(\psi) $ (set to be zero in $ B_R^c $) converges strongly in $ \LL^2(\mathbb{R}^d) $ to $ (-\Delta)^{s/2}(\psi) $, so that the aforementioned weak convergence holds as a consequence of \eqref{eq: def-weak-dom-2} (with $ h=f_\varepsilon $), \eqref{eq: def-weak-dom-3} and the just established finiteness of the r.h.s.~of \eqref{eq: def-weak-dom-4}. In a similar way one can deduce the weak convergence of $ (-\Delta)^{s/2}(u_R) $ to $ (-\Delta)^{s/2}( \mathsf{I}_{d,s} \ast f_\varepsilon ) $ in $ \LL^2(\mathbb{R}^d) $.
We are therefore allowed to pass to the limit in \eqref{eq: def-weak-dom-2} (with $ h=f_\varepsilon $) as $ R \to \infty $ to get
\begin{equation}\label{eq: def-weak-dom-2-bis}
\int_{\mathbb{R}^d} (-\Delta)^{s/2}(\mathsf{I}_{d,s} \ast f_\varepsilon)(x) \, (-\Delta)^{s/2}(\psi)(x) \, \mathrm{d}x  =  \int_{\mathbb{R}^d} f_\varepsilon(x) \, \psi(x) \, \mathrm{d}x \quad \forall \psi \in \HH^s_c(\mathbb{R}^d) \, .
\end{equation}
By letting $ \varepsilon \to 0 $ in \eqref{eq: def-weak-dom-4} we infer that $ (-\Delta)^{s/2}(\mathsf{I}_{d,s} \ast f_\varepsilon) $ converges weakly in $ \LL^2(\mathbb{R}^d) $ to $ (-\Delta)^{s/2}(\mathsf{I}_{d,s} \ast f) $. Indeed, Theorem \ref{lem: lemma-moll} and the identity $ \mathsf{I}_{d,s} \ast f_\varepsilon = (\mathsf{I}_{d,s} \ast f)_\varepsilon $ ensure that 
$$ \lim_{\varepsilon \to 0} \left( \| f_\varepsilon - f \|_{p^\prime,\rho^\prime} + \| \mathsf{I}_{d,s} \ast f_\varepsilon - \mathsf{I}_{d,s} \ast f \|_{p,\rho} \right) = 0 \, . $$ 
Hence, by letting $ \varepsilon \to 0 $ in \eqref{eq: def-weak-dom-2-bis}, we end up with
\begin{equation}\label{eq: def-weak-dom-2-ter}
\int_{\mathbb{R}^d} (-\Delta)^{s/2}(\mathsf{I}_{d,s} \ast f)(x) \, (-\Delta)^{s/2}(\psi)(x) \, \mathrm{d}x  =  \int_{\mathbb{R}^d} f(x) \, \psi(x) \, \mathrm{d}x \quad \forall \psi \in \HH^s_c(\mathbb{R}^d) \cap \LL^\infty(\mathbb{R}^d) \, .
\end{equation} 
We can then plug $ \psi = u_R $ (set to be zero outside $ B_R $) in \eqref{eq: def-weak-dom-2-ter} and let $ R \to \infty $ to obtain
\begin{equation}\label{eq: def-weak-dom-2-quater}
\int_{\mathbb{R}^d} (-\Delta)^{s/2}(\mathsf{I}_{d,s} \ast f)(x) \, (-\Delta)^{s/2}(\mathsf{I}_{d,s} \ast f_\varepsilon)(x) \, \mathrm{d}x  =  \int_{\mathbb{R}^d} f(x) \, (\mathsf{I}_{d,s} \ast f_\varepsilon)(x) \, \mathrm{d}x \, ,
\end{equation} 
so that \eqref{eq: def-weak-dom-1} finally follows by passing to the limit in \eqref{eq: def-weak-dom-2-quater} as $ \varepsilon \to 0 $. The fact that $ \mathsf{I}_{d,s} \ast f \in \dot{\HH}^s(\mathbb{R}^d) $ is just a consequence of the method of proof. Note that, in order to establish the validity of the integration-by-parts formulas \eqref{eq: parti-intro}, it is enough to use \eqref{eq: def-weak-dom-1} with $ f=(-\Delta)^s(v) \pm (-\Delta)^s(w) $.

\smallskip

It remains to prove that, in the case where $ \gamma $ is larger than the spatial dimension $d$, formulas \eqref{eq: parti-intro} always fail. This is due to the presence of nontrivial constants in the space $ X_{p,s,\rho} $ (note that $ \rho $ is in $ \LL^1(\mathbb{R}^d) $). In fact, since $ (-\Delta)^s(1) \equiv 0 $, should \eqref{eq: parti-intro} hold then
\begin{equation}\label{eq: parti-fail}
\int_{\mathbb{R}^d} (-\Delta)^s(v)(x) \, \mathrm{d}x  = 0 \quad \forall v \in X_{p,s,\rho}  \, .
\end{equation}
Given any \emph{nonnegative}, nontrivial function $ \phi \in \mathcal{D}(\mathbb{R}^d) $, let us consider its Riesz potential $ \Phi = \mathsf{I}_{d,s} \ast \phi $. Thanks to Lemma \ref{lem: decay-pot}, it is plain that $ \Phi \in \LL^p_{\rho}(\mathbb{R}^d) $. Moreover, it is apparent that $ (-\Delta)^s(\Phi)=\phi \in \LL^{p^\prime}_{\rho^\prime}(\mathbb{R}^d) $. Hence, $ \Phi \in X_{p,s,\rho} $ and we can plug $ v=\Phi $ in \eqref{eq: parti-fail} to get
\begin{equation*}\label{eq: parti-fail-2}
\int_{\mathbb{R}^d} \phi(x) \, \mathrm{d}x  = 0 \, ,
\end{equation*}
which is absurd unless $ \phi\equiv 0 $.
\end{proof}

Finally, the assertions of Theorem \ref{thm: global-self-adj} for $ d>2s $ and $ \gamma>2s $ are direct consequences of the results just proved and of the method of proof of Theorem \ref{thm: global-self-adj} itself for subcritical-critical powers (see Section \ref{sect: SAIP}).

\bibliographystyle{plainnat}

\end{document}